\documentclass{amsart}

\usepackage{amssymb}
\usepackage{enumitem}
\usepackage{hyperref}
\usepackage{circuitikz}
\usepackage{tikz-cd}
\usepackage{tkz-graph}\usetikzlibrary{decorations.markings}\usetikzlibrary{quotes}
\tikzset{->-/.style={
 decoration={
   markings,
  mark=at position #1 with {\arrow{>}}
   },
   postaction={decorate}
  },
  ->-/.default=0.5 
}

\def\PlusCross{{%
    \setbox0\hbox{$+$}%
    \rlap{\hbox to \wd0{\hss$\times$\hss}}\box0
}}

\newtheorem{thm}{Theorem}[section]
\newtheorem{prop}[thm]{Proposition}
\newtheorem{lem}[thm]{Lemma}
\newtheorem{cor}[thm]{Corollary}

\theoremstyle{definition}

\newtheorem{example}[thm]{Example}

\numberwithin{equation}{section}

\newcommand{\RC}{\operatorname{RC}}

\newcommand{\ol}{\overline}

\begin{document}

\title{Pretzel monoids}

\author{Daniel Heath}
\address{D. Heath, Department of Mathematics, University of Manchester, 
Manchester \linebreak M13 9PL}
\email{daniel.heath-2@manchester.ac.uk}

\author{Mark Kambites}
\address{M. Kambites, Department of Mathematics, University of Manchester, 
Manchester \linebreak M13 9PL}
\email{mark.kambites@manchester.ac.uk}

\author{N\'ora Szak\'acs}
\address{N. Szak\'acs, Department of Mathematics, University of Manchester, 
Manchester \linebreak M13 9PL}
\email{nora.szakacs@manchester.ac.uk}

\begin{abstract}
We introduce an interesting class of left adequate monoids which we call \textit{pretzel monoids}. These, on the one hand, are monoids of birooted graphs with respect to a natural `glue-and-fold' operation, and on the other hand, are shown to be defined in the category of left adequate monoids by a natural class of presentations. They are also shown to be the free idempotent-pure expansions of right cancellative monoids, making them, in some sense, the left adequate analogues of Margolis-Meakin expansions for inverse monoids. The construction recovers the second author's geometric model of free left adequate monoids when the right cancellative monoid is free.
\end{abstract}

\maketitle

\section{Introduction}

\textit{Left adequate monoids} (and semigroups) were introduced and first studied by Fountain in the late 1970s \cite{fountain:classofpp,fountain:adequate}.
 They are (roughly speaking - see Section~\ref{sec:leftad} below for a precise definition) those monoids in which every element shares its right cancellativity properties with an idempotent element.  Although a number of interesting things were
proved about them in the decades that followed (see, for example, \cite{araujo:COASBFS,batdebat:connections,branco:ehresmannadequacy,elqallali:factorisable,elqallali:classofabundant}), their study was until relatively recently hampered by the paucity of easy-to-understand, concrete examples. This began to change with the discovery by the second author of a geometric representation of free left adequate semigroups \cite{kambites:freeleft}, and also of the two-sided analogues \cite{kambites:free}, via birooted, edge-labelled, directed trees under a ``glue-and-retract'' operation. This was inspired by Munn's \cite{munn:freeinv} description (building on earlier work of Scheiblich \cite{scheiblich:freeinv}) of the \textit{free inverse semigroups}, with the replacement of morphisms with retractions providing transition from inverse to adequate semigroups.

Munn's model of free inverse semigroups as birooted trees prompted far-reaching generalisation in inverse semigroup theory, which, rooted in the work of Stephen \cite{stephen:presofinv}, provides a geometric framework for the study of inverse monoid presentations, and describes -- in some sense -- all inverse monoids as monoids of birooted graphs \cite{Pedro}. The first step along this road was the work of Margolis and Meakin \cite{mm:inv}, who introduced an important class of inverse monoids consisting of birooted subgraphs of a given group Cayley graph, which is shown to be the free idempotent-pure expansion of the group.

Given that the free objects in the category of left adequate monoids can similarly be seen as birooted graphs, it is natural next to seek further generalizations to wider classes of left adequate monoids. In general we believe this to be hard --- we are still very far from developing anything akin to Stephen's methods for left adequate monoids in general --- but the main purpose of this paper is to make some first tentative steps in this direction. Specifically we introduce a class of left adequate monoids, consisting of a class of birooted, edge-labelled finite graphs, with respect to certain natural combinatorial operations.  We term the resulting monoids \textit{pretzel monoids}, due to a resemblance between the diagrams we use to represent their elements and certain baked goods. We shall show that pretzel monoids are free, idempotent pure expansions of right cancellative monoids, and thus in some respects provide a left adequate analogue of Margolis-Meakin expansions.

More precisely, each pretzel monoid is constructed from a choice of generators ($X$, say) for a given right cancellative monoid ($C$, say). The geometric realisation, described in Section~\ref{sec:pretzel} below, is by certain edge-labelled directed graphs under a multiplication involving gluing two graphs together, identifying vertices which are joined by paths whose label represents the identity in $C$ (an operation we term \textit{idempath identification}), and then \textit{retracting} in a manner similar to the second author's tree representation \cite{kambites:freeleft} of the free left adequate monoid. We also show that the same monoid admits a presentation (as a left adequate monoid) with generating set $X$, and relations declaring every word representing the identity element in $C$ to be idempotent. It follows that it (analogously to the Margolis-Meakin expansion in the theory of inverse monoids) is an initial object in the category of $X$-generated left adequate monoids having maximum cancellative image $C$ and an idempotent-pure morphism to $C$.

This paper consists of 4 further sections: firstly, in Section~\ref{sec:leftad}, we recall some preliminary definitions of left adequate monoids and deduce some new identities which will prove useful in further study. In Section~\ref{sec:idemiden}, we introduce our new operation of idempath identification on labelled graphs, and establish some technical results about it which will be needed later. In Section~\ref{sec:pretzel}, we define our pretzel monoids as collections of these birooted directed, edge-labelled graphs. Finally, in Section~\ref{sec:presentation}, we provide a presentation for a general pretzel monoid as a presentation of left adequate monoids, and explore the way in which pretzel monoids can be viewed as natural expansions of right cancellative monoids.

\section{Left Adequate Monoids}\label{sec:leftad}

We first recall definitions of left adequate monoids \cite{fountain:classofpp,fountain:adequate}. For a comprehensive introduction to general semigroup and monoid theory, we direct the reader to \cite{howie:fundamentals} and for concepts on universal algebra, we further direct the reader to \cite{burris:universal}.

\subsection{Left Adequacy}

Let $M$ be a monoid. The equivalence relation $\mathcal{R}^*$ is defined on $M$ by $a\mathrel{\mathcal{R}^*}b$ if and only if\[ xa = ya \iff xb = yb\]for all $x,y \in M$. A monoid is called \textit{left abundant} if every $\mathcal{R}^*$-class contains an idempotent. Moreover, a monoid is called \textit{left adequate} if it is left abundant and its set of idempotents, denoted $E(M)$, is commutative. It is quickly seen that any idempotent in an $\mathcal{R}^*$-class of a left adequate monoid $M$ is the unique idempotent in its class. We denote by $a^+$ the unique idempotent in the $\mathcal{R}^*$-class of an element $a \in M$.

Left adequate {monoids} may be considered as algebras of signature {(2,1,0)} {with the associative multiplication and identity element supplemented with
the unary operation $+$}, which relates an element to the unique idempotent with which it shares right cancellation properties. The class of left adequate monoids encompasses the class of \textit{inverse monoids}, with the $+$ operation sending any element $a$ to the (unique) idempotent $aa^{-1}$ in its $\mathcal{R}$-class. It also encompasses the class of \textit{right cancellative} monoids, with the $+$ operation sending every element to the identity. Left adequate monoids form a quasi-variety of algebras in the (2,1,0) signature, with defining quasi-identities \[x^+x = x, \ (x^+y^+)^+ = x^+y^+ = y^+x^+,  \ (xy)^+ = (xy^+)^+,\]\[x^2 = x \rightarrow x = x^+ \textrm { and } xy = zy \rightarrow xy^+ = zy^+\] along with {the associative law and the laws making the constant an identity element}. The variety generated by the class is that of the \textit{left Ehresmann monoids} \cite{branco:ehresmann}. Accordingly, the \textit{free left adequate monoid} and the \textit{free left Ehresmann monoid} on a given set coincide. This common object was given a geometric description by the second author \cite{kambites:freeleft} which we build upon in this paper.

We say a left adequate monoid $M$ is \emph{$X$-generated} if there is a map $\iota \colon X \to M$ such that $\iota(X)$ generates $M$ {as a $(2,1,0)$-algebra}. We do not assume $\iota$ to be an embedding, but when it is, we sometimes naturally associate $X$ with its image. Given a word $w \in X^*$, we denote the image of $w$ under $\iota$ by $[w]_M$.

{There are of course also notions of \textit{right adequate} and \textit{two-sided adequate} monoids (see, for example, \cite{fountain:adequate}) and also
corresponding classes of semigroups, but we do not consider these here.}

\subsection{\texorpdfstring{$X$-graphs}{X-graphs}}

Fix a non-empty set $X$. An \textit{$X$-graph} $G$ (or simply \textit{graph} if $X$ is clear) is a finite birooted digraph, with vertices $\mathrm{V}(G)$ including a \textit{start} vertex $\alpha(G)$, which we represent diagrammatically by {\Large$+$} and \textit{end} vertex $\omega(G)$, diagrammatically {\Large$\times$}, with edges $\mathrm{E}(G)$ labelled by elements of $X$, such that any vertex is reachable via some directed path from the start vertex. An \textit{$X$-tree} (or simply \textit{tree}) is an $X$-graph whose underlying, undirected graph is a tree with all edges oriented away from the {start vertex}. {By the \textit{trunk} of an $X$-tree we mean the unique (necessarily directed) simple path from the start vertex to the end vertex.} A \textit{leaf} of a tree is a vertex $l$ for which there is no edge with initial vertex $l$. Figure \ref{fig:trees} shows three examples of $\{x,y\}$-graphs.

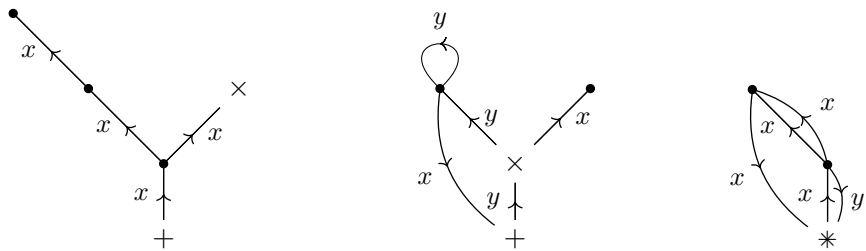
\begin{figure}[ht]
    \centering
    {
    \begin{tikzpicture}
        \GraphInit[vstyle=Empty]
        \SetVertexSimple[MinSize = 1pt]
        \SetUpEdge[lw = 0.5pt]
        \tikzset{EdgeStyle/.style={->-}}
        \tikzset{VertexStyle/.append style = {minimum size = 3pt, inner sep = 0pt}}
        \SetVertexNoLabel
        \SetGraphUnit{2}
        
        \node (A) at ( 0,0) {\large$+$};
        \Vertex[x=0,y=1]{C1}
        \Vertex[x=-1,y=2]{L1}
        \Vertex[x=-2,y=3]{L2}
        \node (R1) at ( 1,2) {\large$\times$};
                  
        \Edge(A)(C1)\draw (A) -- (C1) node [midway, left=2pt] {$x$};
        \Edge(C1)(R1)\draw (C1) -- (R1) node [midway, right=2pt] {$x$};
        \Edge(C1)(L1)\draw (C1) -- (L1) node [midway, left=2pt] {$x$};
        \Edge(L1)(L2)\draw (L1) -- (L2) node [midway, left=2pt] {$x$};
        
    \end{tikzpicture}
    }\hspace{0.1\textwidth}
    {
    \begin{tikzpicture}
        \GraphInit[vstyle=Empty]
        \SetVertexSimple[MinSize = 1pt]
        \SetUpEdge[lw = 0.5pt]
        \tikzset{EdgeStyle/.style={->-}}
        \tikzset{VertexStyle/.append style = {minimum size = 3pt, inner sep = 0pt}}
        \SetVertexNoLabel
        \SetGraphUnit{2}
        
        \node (A) at ( 0,0) {\large$+$};
        \Vertex[x=-1,y=2]{L1}
        \node (C1) at ( 0,1) {\large$\times$};
        \Vertex[x=1,y=2]{R1} 
        
        \Edge(A)(C1)\draw (A) -- (C1) node [midway, left=2pt] {$y$};
        \Loop[style={->-}, dir=NO, dist=30pt](L1)\node (l) at ( -1,2.9) {$y$};
        \Edge(C1)(R1)\draw (C1) -- (R1) node [midway, right=2pt] {$x$};
        \Edge(C1)(L1)\draw (C1) -- (L1) node [midway, right=2pt] {$y$};
        \Edge[style={bend right}](L1)(A)\node (l) at ( -1.2,0.8) {$x$};
        
    \end{tikzpicture}
    }\hspace{0.1\textwidth}
    {
    \begin{tikzpicture}
        \GraphInit[vstyle=Empty]
        \SetVertexSimple[MinSize = 1pt]
        \SetUpEdge[lw = 0.5pt]
        \tikzset{EdgeStyle/.style={->-}}
        \tikzset{VertexStyle/.append style = {minimum size = 3pt, inner sep = 0pt}}
        \SetVertexNoLabel
        \SetGraphUnit{2}
        
        \node (A) at ( 0,0) {\PlusCross};
        \Vertex[x=0,y=1]{C1}
        \Vertex[x=-1,y=2]{L1}
                  
        \Edge(A)(C1)\draw (A) -- (C1) node [midway, left=2pt] {$x$};
        \Edge[style={bend right}](C1)(L1)\node (l) at ( 0,1.8) {$x$};
        \Edge(C1)(L1)\draw (C1) -- (L1) node [midway, left=2pt] {$x$};
        \Edge[style={bend right}](L1)(A)\node (l) at ( -1.2,0.8) {$x$};
        \Edge[style={bend left}](C1)(A)\node (l) at ( 0.4,0.5) {$y$};
            
    \end{tikzpicture}}
    \caption{Some examples of $\{x,y\}$-graphs. The leftmost graph is an $\{x,y\}$-tree. The rightmost graph $G$ has $\alpha(G) = \omega(G)$.}
    \label{fig:trees}
\end{figure}

For an edge $e$ of an $X$-graph, we denote by $\alpha(e)$ the initial vertex of $e$, $\lambda(e)$ the $X$-label of $e$, and $\omega(e)$ terminal vertex of $e$. We say $e$ is \textit{incident} to a vertex $v$ if $v = \alpha(e)$ or $v = \omega(e)$. We often represent an edge $e$ by the diagram \[\alpha(e) \xrightarrow{\lambda(e)} \omega(e).\] {By a (\textit{directed}) \textit{path} in a graph, we mean a traversal of the graph by a sequence of edges which may repeat both edges and vertices. By a \textit{simple} path, we mean a path which does not repeat vertices.} A \textit{cycle} is a closed path, and a \textit{simple cycle} is a closed simple path.

We extend the maps $\alpha$, $\lambda$ and $\omega$ to (directed) paths in the expected way, taking concatenation of labels for $\lambda$. We represent a path $\pi$ with edges $e_1,e_2,\dots,e_k$ by the diagram\[\alpha(e_1) \xrightarrow{\lambda(e_1)} \omega(e_1) = \alpha(e_2) \xrightarrow{\lambda(e_2)} \omega(e_2) = \alpha(e_3) \xrightarrow{e_3} \dots \xrightarrow{e_k} \omega(e_k)\]or simply by\[\alpha(\pi) \xrightarrow{\lambda(\pi)} \omega(\pi).\]

{We write $X^*$ for the set of words over $X$, and say} a word $w \in X^*$ is \textit{readable} in $G$ from a vertex $p$ to a vertex $q$ if there exists a path $p \xrightarrow{w} q$.

A \textit{subgraph} of an $X$-graph $G$ is a subset $V'$ of vertices and $E'$ of edges such that the digraph with vertex set $V'$ and edge set $E'$ has some vertex $\alpha$ such that all vertices are reachable via a directed path from $\alpha$. A \textit{subtree} is a subgraph in which the underlying, undirected graph is a tree. We call a subgraph or subtree of $G$ \textit{rooted} if it contains $\alpha(G)$, and \textit{birooted} if it contains both $\alpha(G)$ and $\omega(G)$.

{Given a vertex $u$ of an $X$-tree, the \textit{cone} of $u$ in $T$, denoted $\mathrm{Cone}_T(u)$, is the largest subtree of $T$ containing all vertices reachable from $u$.} The \textit{height} of an $X$-tree $T$ (or cone) is given by the length of a longest path in $T$ starting from $\alpha(T)$.

A \textit{morphism} of $X$-graphs is a map between $X$-graphs sending edges to edges and vertices to vertices, such that the start [resp. end] vertex is sent to the start [resp. end] vertex, and which respect edge labellings and incidence of edges in the expected way. An idempotent endomorphism on a graph is called a \textit{retract}. A graph is called \textit{retracted} if it admits no non-identity retracts. {Each} finite graph $G$ admits a {retracted image which is unique up to isomorphism (although not
unique as a} {birooted} {subgraph); this is often called the \textit{core} \cite{hell:core} and we denote it $\overline{G}$}.

\begin{lem}\label{lem:retractfrommorph}
     Let $S$ be an $X$-graph, $T$ be a birooted subgraph of $S$ and suppose there exists a morphism $\phi: S \to T$. Then $\overline{S}\cong \overline{T}$.
 \end{lem}

 \begin{proof}
 Since $T$ is a subgraph of $S$, we may view $\phi$ as a map from $S$ to itself, and hence compose it with itself. Since $S$ is finite, there exists an $n \in \mathbb{N}$ such that $\phi^n$ is idempotent, in other words, a retract of $S$ with image contained in $T$. Moreover, it is easily seen that $\phi^n|_T$ induces a retract on $T$. We now show that $\phi^n(S) = \phi^n(T)$. Since $T$ is a subgraph of $S$, clearly $\phi^n(T) \subseteq \phi^n(S)$. Conversely, since $\phi^n$ is idempotent, if $y \in \phi^n(S)$, then $y \in \phi^{2n}(S) = \phi^{n}\phi^{n}(S) \subseteq \phi^n(T)$.
 
 Thus by uniqueness and confluence of retracts \cite{hell:core}, we have that \[\overline{S} \cong \overline{\phi^n(S)} \cong \overline{\phi^n(T)} \cong \overline{T}\]as required.
 \end{proof}

The following is the main result from \cite[Theorem 3.16]{kambites:freeleft}.

\begin{thm}
    The free left adequate monoid $\mathrm{FLAd}(X)$ is given by the set of all isomorphism classes of retracted $X$-trees, with multiplication given by (the class of) gluing representatives start-to-end then retracting, and $+$-operation given by moving the end vertex to the start and then retracting. This structure is generated by the set of \textit{base trees} given by single-edge trees with distinct start and end vertices labelled $x$ for each $x \in X$ (we identify this set of trees with $X$ itself).
\end{thm}
 
 For reasons hinted at by the above result, we generally work with $X$-graphs up to isomorphism. For notational simplicity we will (where no possible ambiguity arises)
 identify $X$-graphs with their isomorphism types, and hence write that two $X$-graphs are equal when we mean formally only that they are isomorphic.

 With this viewpoint, we may consider elements of general left adequate monoids as classes of $X$-trees. Similar to our notation for words, for an $X$-tree $T$ and $X$-generated left adequate monoid $M$, we denote by $[T]_M$ the image of {$\overline{T}$} under the surjective $(2,1,0)$-morphism $\textrm{FLAd}(X) \to M$ which maps $x \mapsto x$ for all $x \in X$. We say that two trees $T_1$ and $T_2$ are \textit{equal in $M$} if $[T_1]_M = [T_2]_M$. We may define a multiplication and a unary operation on all (not necessarily retracted) $X$-trees using the glue and $+$ operations without retraction. Whilst the monoid obtained in this way is not left adequate, it has the property that trees equal in this monoid are equal in all left adequate monoids.

\subsection{Identities} We list some properties of left adequate monoids which we will use without further reference (see \cite[Proposition 2.1]{kambites:free}). We invite the reader to view these identities geometrically using $X$-trees.

\begin{lem}\label{lem:ladidentities}
    Let $M$ be a left adequate monoid and let $a,b,e\in M$ with $e$ idempotent. Then $e^+ = e$, $(ab)^+ = (ab^+)^+$, $a^+a = a$, $ea^+ = (ea)^+$ and $a^+(ab)^+ = (ab)^+$.
\end{lem}

We now provide some new identities which will prove useful in the study of left adequate monoids.

\begin{lem}\label{lem:leftadidents}
    Let $M$ be a left adequate monoid and $x,y,z \in M$ with $xy$ idempotent. Then:
    \begin{enumerate}[label={(\roman*)},itemindent=0.5em]
        \item $yxyx \in E(M)$;\label{prop:yxyx}
        \item $xy^+ = xyx$;\label{prop:xy+}
        \item $xz^+y = xy(xz)^+.$ In particular, $xz^+y$ is idempotent.\label{prop:xz+y}
    \end{enumerate}
\end{lem}
\begin{proof}
    Throughout, we suppose $xy$ idempotent, that is $(xy)^+ = xy = xyxy$, and recall that idempotents commute.
    \begin{enumerate}[label={(\roman*)},itemindent=0.5em]
        \item $yxyx = y(xy)x = y(xy)(xy)(xy)x = yxyxyxyx = (yxyx)(yxyx)$.
        \item We have\begin{align*}
            xy^+ &= (xy)^+(xy)^+xy^+ = xyxyxy^+ = x(yxyx)y^+ = xy^+(yxyx)\\
            &= x(yxyx) = (xy)(xy)x = xyx.
        \end{align*}
        \item First note that $xz^+y = xz^+y^+y = xy^+z^+y = xyxz^+y$ by \ref{prop:xy+}. Thus \[(xz^+y)^2 = (xz^+y)xy(xz^+y) = xz^+(yxyx)z^+y = xyxyxz^+y = xyxz^+y = xz^+y\]
and hence $xz^+y \in E(M)$. Moreover, \[(xz^+y)^+ = (xz^+y^+)^+ = (xy^+z^+)^+ = (xy^+z)^+ = (xyxz)^+ = xy(xz)^+.\] 
where the penultimate inequality is from \ref{prop:xy+}. Since $xz^+y$ is idempotent, we have that $xz^+y = (xz^+y)^+ = xy(xz)^+$.
    \end{enumerate}
\end{proof}

In fact, a more general result than Lemma \ref{lem:leftadidents}\ref{prop:xz+y} holds. 

\begin{thm}\label{thm:idemtrunk}
Let $M$ be a left adequate monoid and $x_1,\dots,x_k \in M$ such that $x_1\cdots x_k$ idempotent. For any $u_1,u_2,\dots, u_{k-1} \in M$, we have:
\begin{equation}\label{eq:idemtrunk}
\begin{split}
&x_1{u_1}^+x_2{u_2}^+\cdots x_{k-1}{u_{k-1}}^+x_k =\\ &(x_1\cdots x_k)(x_1u_1)^+(x_1x_2u_2)^+\cdots(x_1\cdots x_{k-1}u_{k-1})^+\end{split}
\end{equation}In particular, $x_1{u_1}^+x_2{u_2}^+\cdots x_{k-1}{u_{k-1}}^+x_k$ is idempotent.

\end{thm}

\begin{proof}
We proceed by strong induction. If $k = 1$ then there is nothing to prove. If $k=2$ then this is exactly the result of Lemma \ref{lem:leftadidents}\ref{prop:xz+y}.

Now suppose the identity $({\ref{eq:idemtrunk}})$ holds for integers $0 < k < N$ for some $N \geq 3$. Let $x_1,\dots,x_N$ and $u_1,\dots,u_{N-1}$ be given such that $x_1\cdots x_N$ is an idempotent of $M$. For notation, we denote $X_\ell := x_1x_2\cdots x_\ell$ for $1 \leq \ell \leq N$.

Consider the product $(x_1x_2)x_3\cdots x_N$ of $M$. By assumption, this is an idempotent product of $N-1$ elements of $M$. Thus by our inductive assumption, we have that
\begin{equation}\label{eq:idemtrunk1}(x_1x_2){u_2}^+x_3{u_3}^+\cdots x_{N-1}{u_{N-1}}^+x_N = (X_N)(X_2u_2)^+\cdots(X_{N-1}u_{N-1})^+\end{equation}In particular, $(x_1x_2){u_2}^+x_3{u_3}^+\cdots x_{N-1}{u_{N-1}}^+x_N$ is an idempotent of $M$. Firstly, by right-multiplying \eqref{eq:idemtrunk1} by $(x_1u_1)^+$, we have that\begin{equation}\label{eq:idemtrunk2}
\begin{split}
(x_1x_2){u_2}^+x_3{u_3}^+\cdots x_{N-1}{u_{N-1}}^+x_N(x_1u_1)^+ =&\ (X_N)(X_2u_2)^+\cdots\\
 &\quad (X_{N-1}u_{N-1})^+(x_1u_1)^+.\end{split}
\end{equation}By commuting idempotents on the right-hand side of \eqref{eq:idemtrunk2}, we obtain exactly the right-hand side of $({\ref{eq:idemtrunk}})$.

Now, we appeal to Lemma \ref{lem:leftadidents}\ref{prop:xz+y}. Since the left-hand side of \eqref{eq:idemtrunk1} is idempotent, we may insert the idempotent ${u_1}^+$ after $x_1$ to obtain\begin{equation}\label{eq:idemtrunk3}
\begin{split}
x_1{u_1}^+x_2{u_2}^+\cdots x_{N-1}{u_{N-1}}^+x_N =&\ (x_1x_2){u_2}^+x_3{u_3}^+\cdots \\
 &\quad x_{N-1}{u_{N-1}}^+x_N(x_1u_1)^+.\end{split}
\end{equation}The left-hand side of \eqref{eq:idemtrunk3} is exactly the left-hand side of $({\ref{eq:idemtrunk}})$, and the right-hand side of \eqref{eq:idemtrunk3} is exactly the left-hand side of \eqref{eq:idemtrunk2}. Hence we have total equality: the left-hand side of $({\ref{eq:idemtrunk}})$ is the left-hand side of \eqref{eq:idemtrunk3}, which is equal to the right-hand side of \eqref{eq:idemtrunk2}, which is in turn equal to the right-hand side of $({\ref{eq:idemtrunk}})$. Thus our equation $({\ref{eq:idemtrunk}})$ holds. The full result therefore follows by induction.\end{proof}

Theorem \ref{thm:idemtrunk} allows us the following important corollary.

\begin{cor}\label{cor:idemtrunk}
    Let $M$ be an $X$-generated left adequate monoid and $T$ be an $X$-tree. Let $\Theta(T)$ be the birooted subtree of $T$ consisting exactly of the trunk of $T$. Then $[\Theta(T)]_M \in E(M)$ implies $[T]_M \in E(M)$.
\end{cor}
\begin{proof}
Suppose $\Theta(T)$ has label $x_1\cdots x_k$ with $[x_1\cdots x_k]_M \in E(M)$. Then $T$ may be written in the form\[T = {u_0}^+x_1{u_1}^+x_2{u_2}^+\cdots x_{k-1}{u_{k-1}}^+x_k{u_k}^+\]for some $X$-trees $u_i$. By Theorem \ref{thm:idemtrunk}, we have\[ [T]_M = [{u_0}^+(x_1\cdots x_k)(x_1u_1)^+(x_1x_2u_2)^+\cdots(x_1\cdots x_{k-1}u_{k-1})^+{u_k}^+]_M\]which is a product of idempotents and hence idempotent.
\end{proof}

We now consider an $X$-generated left adequate monoid $M$ and an $X$-tree $T$. Suppose $T$ contains some path labelled by a word $x_1x_2\cdots x_k$ such that $[x_1x_2\cdots x_k]_M$ is an idempotent of $M$. Our goal is to show that $[T]_M$ is equal to the $M$-value of the tree obtained by `moving' a subtree through the path. 

\begin{thm}\label{thm:movedown}
    Let $M$ be an $X$-generated left adequate monoid and $T$ be an $X$-tree. Suppose $T$ contains a path $w$ labelled $x_1x_2\cdots x_k$ such that \[[x_1x_2\cdots x_k]_M \in E(M).\] Suppose $V$ is a subtree of $\mathrm{Cone}_T(\omega(w))$, containing the vertex $\omega(w)$, whose erasure from $T$ would not disconnect the underlying (undirected) tree of $T$. Let $T'$ be the tree obtained by taking $T$ with $V$ erased, and gluing a copy of $V$ to $\alpha(w)$ (via the vertex $\omega(w)$ in $V$). Then $[T]_M = [T']_M$.
\end{thm}

\begin{proof}
    We first consider the case when $V$ contains the endpoint. Since $V$ contains the endpoint and contains $\omega(w)$, we must have that $w$ lies entirely on the trunk of $T$. Hence $T$ is of the form $Ux_1{u_1}^+x_2{u_2}^+\cdots x_{k-1}{u_{k-1}}^+x_k{u_k}^+V$ for some $X$-trees $U,u_1,\dots,u_{k}$. By Corollary \ref{cor:idemtrunk}, $[x_1{u_1}^+x_2{u_2}^+\cdots x_{k-1}{u_{k-1}}^+x_k{u_k}^+]_M$ is an idempotent of $M$, and thus \[[T]_M = [U(x_1{u_1}^+x_2{u_2}^+\cdots x_{k-1}{u_{k-1}}^+x_k{u_k}^+)^+V]_M.\] The $X$-tree $U(x_1{u_1}^+x_2{u_2}^+\cdots x_{k-1}{u_{k-1}}^+x_k{u_k}^+)^+V$ is exactly $T'$, and so we are done.

Now suppose that $V$ does not contain the endpoint of $T$. Let $x_i$ be the label of the first edge of $w$ which is not on the trunk of $T$. Then $T$ is of the form\[Ux_1{u_1}^+x_2{u_2}^+\cdots x_{i-1}{u_{i-1}}^+(x_i{u_i}^+x_{i+1}{u_{i+1}}^+\cdots x_{k-1}{u_{k-1}}^+x_kQ^+V)^+W\] for some $X$-trees $U,Q,u_1,\dots,u_{k-1},W$ (note that the subtree $V$ does not contain the endpoint of $T$). For notational simplicity, write $w_1 = x_1{u_1}^+x_2{u_2}^+\cdots x_i{u_i}^+$ and $w_2 = x_{i+1}{u_{i+1}}^+\cdots x_{k-1}{u_{k-1}}^+x_k$. Then in $M$, since the retraction map and $[\cdot]$ are morphisms, we have\begin{align*}
[T]_M &= [Uw_1(w_2Q^+V)^+W]_M && &&&\\
  & = [Uw_1(w_2Q^+V^+)^+W]_M && \textrm{\small(by Lemma \ref{lem:ladidentities})}&&& \\
  & = [Uw_1w_2Q^+V^+w_1W]_M && \textrm{\small(by Lemma \ref{lem:leftadidents}\ref{prop:xy+}: }[w_1w_2Q^+V^+]_M \in E(M) \textrm{\small)}&&& \\
  & = [UV^+w_1w_2Q^+w_1W]_M && \textrm{\small(by commuting idempotents)}&&& \\
  & = [UV^+w_1(w_2Q^+)^+W]_M && \textrm{\small(by Lemma \ref{lem:leftadidents}\ref{prop:xy+}: }[w_1w_2Q^+]_M \in E(M) \textrm{\small)}&&&         
\end{align*}

Recalling our definitions of $w_1$ and $w_2$, we have\[[T]_M = [UV^+x_1{u_1}^+x_2{u_2}^+\cdots x_i{u_i}^+(x_{i+1}{u_{i+1}}^+\cdots x_{k-1}{u_{k-1}}^+x_kQ^+)^+W]_M.\]This $X$-tree is exactly $T'$.\end{proof}

Theorem \ref{thm:movedown} tells us that we can move subtrees of cones down through paths labelled by idempotents (provided we do not disconnect the tree) and maintain the value of $[T]_M$. By noting that we may first glue an extra copies of $V$ at $\omega(w)$ without changing $[T]_{\mathrm{FLAd}(X)}$, a clear corollary of Theorem \ref{thm:movedown} is that we may similarly copy these subtrees, up or down, through idempotent-labelled paths.

\begin{cor}\label{cor:movethru}

    Suppose an $X$-tree $T$ contains a path $w$ labelled by $x_1x_2 \cdots x_k$ such that $[x_1x_2\cdots x_k]_M \in E(M)$. Suppose $V$ is a subtree of $\mathrm{Cone}_T(\omega(w))$ $\mathrm{[}$resp. $\mathrm{Cone}_T(\alpha(w))\mathrm{]}$, containing the vertex $\omega(w)$ $\mathrm{[}$resp. $\alpha(w)\mathrm{]}$ and not containing the vertex $\alpha(T)$, whose erasure would not disconnect the underlying tree of $T$. Then $[T]_M$ is equal in $M$ to $[T']_M$ where $T'$ is the tree $T$ with $V$ copied to $\alpha(w)$ $\mathrm{[}$resp. $\omega(w)\mathrm{]}$ (including moving the endpoint if $V$ contained it). Moreover, if $V$ does not contain any edge of $w$, we instead can move $V$ to $\alpha(w)$ $\mathrm{[}$resp. $\omega(w)\mathrm{]}$ rather than copying.
\end{cor}

\begin{cor}\label{cor:idemtoleaves}
    For any $X$-generated left adequate monoid $M$ and $X$-tree $T$, there exists a tree $T' \in \textrm{FLAd}(X)$ such that $[T]_M = [T']_M$, and in which all paths labelled by idempotents of $M$ go to leaves.
\end{cor}

\begin{proof}
    For any idempotent-labelled path $w$ in $T$, we may move the entire cone rooted at $\omega(w)$ to $\alpha(w)$ without changing the value in $M$ by Theorem \ref{thm:movedown}. This ensures that the path labelled $w$ terminates at a leaf. We claim that continually performing this action is a terminating process. Indeed, each move maintains the number of vertices of $T$ but strictly increases the number of leaf vertices of $T$. Since $T$ has finitely many vertices, this process must terminate: the result is a tree in which idempotent-labelled paths go to leaves. Retracting this tree then yields a tree $T' \in \textrm{FLAd}(X)$ with the required properties.
\end{proof}

\section{Idempath Identification}\label{sec:idemiden}

As in many areas of semigroup theory, understanding idempotents is critical to understanding the structure of left adequate monoids. We now consider presentation of left adequate monoids where define our set of idempotents. This is done in a similar spirit to that of Margolis-Meakin expansions for $E$-unitary inverse monoids \cite{mm:inv}; we explore this connection further in Section~\ref{sec:presentation}.

Let $X$ be any set and let $E \subseteq X^*$. Consider the left adequate monoid given by the presentation\[\textrm{LAd}\left\langle X\ |\ w^+ = w \textrm{ for } w \in E\right\rangle.\]The corresponding notion for Margolis-Meakin expansions is to take $E$ to be the set of words representing the identity element in some fixed $X$-generated group. In the left adequate setting, it is natural (see Section~\ref{sec:presentation}) to take $E$ to be the words representing the identity element in some fixed $X$-generated right cancellative monoid $C$. This leads us to our definition of \textit{idempath identification}.

\subsection{Idempath Identification} Let $G$ be an $X$-graph and let $C$ be an $X$-generated right cancellative monoid. A path in $G$ is called a $C$-\textit{idempath} (or simply an \textit{idempath} when $C$ is clear) if $[w]_C = 1$. We consider the empty path (labelled $\epsilon$) to have $[\epsilon]_C = 1$ (and thus is always an idempath).

A ($C$-)\textit{idempath identification} on $G$ is the process of identifying the vertices which are the termini of an idempath in $G$. Some literature refers to this action as the \textit{fusion} of the two vertices \cite{ramakrishnan:fusion}. Note that this is different than the process of \textit{contraction}; we do not delete any edges of $G$. An idempath identification $G$ is a morphism of $X$-graphs from $G$ to the image after the identification, and therefore it preserves expected adjacency and connectivity properties.

A $X$-graph $H$ is called a \textit{($C$-idempath) descendant} of an $X$-graph $G$ if $H$ is obtainable via successive idempath identifications, starting with the graph $G$. We denote by $D(G)$ the set of all descendants of $G$. Formally, we consider the set of vertex pairs $V(G) \times V(G)$ to act by partial maps on elements of $D(G)$ where $(a,b)H$ is the $X$-graph obtained by identifying the natural images of the vertices $a,b$ in $H$; accordingly we specify that this action is only defined when there is an idempath readable from natural images of $a$ to $b$ in $H$. For ease of notation, we often identify these images of $a$ and $b$ with $a$ and $b$ themselves.

\begin{example}
    Consider the group $C = C_3 = \textrm{Mon}\langle x\ |\ x^3 = 1 \rangle$. In an $\{x\}$-graph, a path labelled $x^3$ is a $C$-idempath. Figure \ref{fig:x3idems} shows a process of successive idempath identifications to the tree $x(x^2)^+x \in \textrm{FLAd}(\{x\})$. Note that the second identification is only possible once the first identification is made.
    \begin{figure}[ht]
        \centering
        {
        \begin{tikzpicture}
            \GraphInit[vstyle=Empty]
            \SetVertexSimple[MinSize = 1pt]
            \SetUpEdge[lw = 0.5pt]
            \tikzset{EdgeStyle/.style={->-}}
            \tikzset{VertexStyle/.append style = {minimum size = 3pt, inner sep = 0pt}}
            \SetVertexNoLabel
            \SetGraphUnit{2}
            
            \node (A) at ( 0,0) {\large$+$};
            \Vertex[x=0,y=1]{C1}
            \Vertex[x=-1,y=2]{L1}
            \Vertex[x=-2,y=3]{L2}
            \node (R1) at ( 1,2) {\large$\times$};
                      
            \Edge(A)(C1)\draw (A) -- (C1) node [midway, left=2pt] {$x$};
            \Edge(C1)(R1)\draw (C1) -- (R1) node [midway, right=2pt] {$x$};
            \Edge(C1)(L1)\draw (C1) -- (L1) node [midway, left=2pt] {$x$};
            \Edge(L1)(L2)\draw (L1) -- (L2) node [midway, left=2pt] {$x$};
            
        \end{tikzpicture}
        }\hspace{0.1\textwidth}
        {
        \begin{tikzpicture}
            \GraphInit[vstyle=Empty]
            \SetVertexSimple[MinSize = 1pt]
            \SetUpEdge[lw = 0.5pt]
            \tikzset{EdgeStyle/.style={->-}}
            \tikzset{VertexStyle/.append style = {minimum size = 3pt, inner sep = 0pt}}
            \SetVertexNoLabel
            \SetGraphUnit{2}
            
            \node (A) at ( 0,0) {\large$+$};
            \Vertex[x=0,y=1]{C1}
            \Vertex[x=-1,y=2]{L1}
            \node (R1) at ( 1,2) {\large$\times$};
                      
            \Edge(A)(C1)\draw (A) -- (C1) node [midway, left=2pt] {$x$};
            \Edge(C1)(R1)\draw (C1) -- (R1) node [midway, right=2pt] {$x$};
            \Edge(C1)(L1)\draw (C1) -- (L1) node [midway, left=2pt] {$x$};
            \Edge[style={bend right}](L1)(A)\node (l) at ( -1.2,0.8) {$x$};
            
        \end{tikzpicture}
        }\hspace{0.1\textwidth}
        {
        \begin{tikzpicture}
            \GraphInit[vstyle=Empty]
            \SetVertexSimple[MinSize = 1pt]
            \SetUpEdge[lw = 0.5pt]
            \tikzset{EdgeStyle/.style={->-}}
            \tikzset{VertexStyle/.append style = {minimum size = 3pt, inner sep = 0pt}}
            \SetVertexNoLabel
            \SetGraphUnit{2}
            
            \node (A) at ( 0,0) {\large$+$};
            \Vertex[x=0,y=1]{C1}
            \node (L1) at ( -1,2) {\large$\times$};
                      
            \Edge(A)(C1)\draw (A) -- (C1) node [midway, left=2pt] {$x$};
            \Edge[style={bend right}](C1)(L1)\node (l) at ( 0,1.8) {$x$};
            \Edge(C1)(L1)\draw (C1) -- (L1) node [midway, left=2pt] {$x$};
            \Edge[style={bend right}](L1)(A)\node (l) at ( -1.2,0.8) {$x$};
            
        \end{tikzpicture}}
        \caption{Successive $C_3$-idempath identifications on $x(x^2)^+x$.}
        \label{fig:x3idems}
    \end{figure}
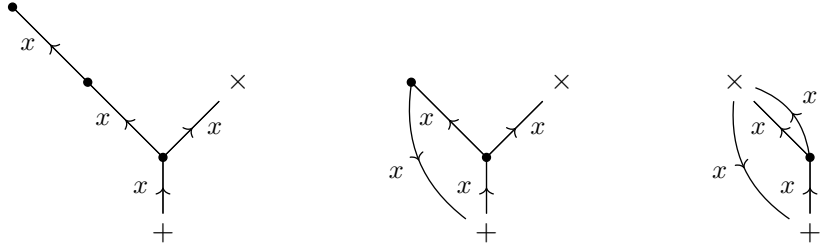    
\end{example}

Our goal is to show that any $X$-graph $G$ has a unique \textit{idempath identified} descendant, that is a graph where no non-trivial idempath identifications are possible (equivalently if every idempath is a cycle).

\begin{lem}\label{lem:descendant}
Let $C$ be an $X$-generated right cancellative monoid. Let $G$ be an $X$-graph. Suppose $(a,b)G$ is defined, $(c,d)G$ is defined, and $H \in D(G)$. Then
\begin{enumerate}[label={(\roman*)},itemindent=0.5em]
    \item $(a,b)H$ is defined;\label{lem:descendant1}
    \item $(a,b)(c,d)G = (c,d)(a,b)G$;\label{lem:descendant2}
    \item $(a,b)H \in D((a,b)G)$.
\end{enumerate}
\end{lem}

\begin{proof}
Recall that $(p,q)K$ is only defined when there is an idempath readable from $p$ to $q$ in the graph $K$.
    \begin{enumerate}[label={(\roman*)},itemindent=0.5em]
        \item Since $(a,b)G$ is defined, there is an idempath labelled by $x_1 \cdots x_n$ from $a$ to $b$ in $G$. This word is still readable in $H$ from $a$ to $b$ in $H$. Hence $(a,b)H$ is defined.
        \item By part \ref{lem:descendant1}, $(c,d)(a,b)G$ and $(a,b)(c,d)G$ are both defined. If $\{a,b\} \cap \{c,d\} = \emptyset$, then clearly $(a,b)(c,d)G = (c,d)(a,b)G$. Otherwise, if $\{a,b\} \cap \{c,d\} \neq \emptyset$, then all of $a,b,c$ and $d$ are identified in the graphs $(a,b)(c,d)G$ and $(c,d)(a,b)G$. Since both graphs are the unique graph obtained by identifying all four vertices, it follows that $(a,b)(c,d)G = (c,d)(a,b)G$.
        \item Since $H \in D(G)$, $H$ is of the form $H = (u_n,v_n)\dots(u_1,v_1)G$ for vertices $u_i,v_i$. Then using part \ref{lem:descendant2} $n$ times, we have
    \[ (a,b)H = (a,b)(u_n,v_n)\dots(u_1,v_1)G = (u_n,v_n)\dots(u_1,v_1)(a,b)G \in D((a,b)G)\]as required.
    \end{enumerate}\end{proof}

\begin{thm}\label{thm:uniquetilde}
Let $C$ be an $X$-generated right cancellative monoid. Let $G$ be an $X$-graph. Then there is a unique graph which is an idempath identified descendant of $G$.
\end{thm}

\begin{proof}
Since $G$ is finite, we may only apply a finitely long sequence of non-trivial idempath identifications to $G$ as each strictly reduces the number of vertices. By always performing any non-trivial idempath identifications we can, we may obtain some graph which has no non-trivial identifications. This graph must be an idempath identified descendant of $G$, and so $G$ has some idempath identified descendant $Y \in D(G)$.

Now let $H \in D(G)$. Then there exists some $n \geq 0$ such that \[H = (u_n,v_n)\dots(u_1,v_1)G\]for some vertices $u_i$ and $v_i$ of $G$ for $1 \leq i \leq n$. We claim that $Y \in D(H)$. We perform induction on $n$.

If $n=1$, then $H = (u_1,v_1)G$ and so by Lemma \ref{lem:descendant}, $(u_1,v_1)Y \in D(H)$. But $Y$ is idempath identified, so as this is defined, we must have $u_1 = v_1$ in $Y$. Hence $Y = (u_1,v_1)Y \in D(H)$ as required. Now suppose that the statement holds for some integer $k-1 \geq 1$. Then suppose $H = (u_k,v_k)\dots(u_1,v_1)G$. By the inductive hypothesis, $Y$ is a descendant of $(u_{k-1},v_{k-1})\dots(u_1,v_1)G$. Hence by Lemma \ref{lem:descendant}, $(u_k,v_k)Y$ is a descendant of $(u_k,v_k)(u_{k-1},v_{k-1})\dots(u_1,v_1)G = H$. As above, since $Y$ is idempath identified, $(u_k,v_k)Y = Y$. Hence $Y$ is a descendant of $H$ as required. Our claim therefore holds by induction.

It remains to show uniqueness of $Y$. Indeed, if $Y_1$ and $Y_2$ are idempath identified descendants of $G$ then by the inductive argument above, $Y_1$ is a descendant of $Y_2$ and vice versa. Hence clearly $Y_1 \cong Y_2$ as $X$-graphs.
\end{proof}

We denote the unique $C$-idempath identified descendant of $G$ by $\widetilde{G}$. As idempath identifications are morphisms, the induced map $\sim\ :\ G \mapsto \widetilde{G}$ is also a morphism of $X$-graphs. Given an edge $e$ [resp. vertex $v$] of any descendant of $G$, we denote the image of $e$ [resp. $v$] under the $\sim$ map by $\widetilde{e}$ [resp. $\widetilde{v}$].

\begin{lem}\label{lem:idempathsSCC}
    Let $G$ be an $X$-graph. Let $w$ be an idempath in $G$ traversing vertices $v_0,v_1,\dots,v_k$. Then the vertices $\widetilde{v_0},\widetilde{v_1},\dots,\widetilde{v_k}$ all lie in the same strongly connected component of $\widetilde{G}$.
\end{lem}

\begin{proof}
    There exists a valid idempath identification $(v_0,v_k)$ on $G$. In $(v_0,v_k)G$, the images of the vertices $v_i$ all lie on a directed cycle, in particular in the same strongly connected component. It follows that they remain on a directed cycle in all descendants of $(v_0,v_k)G$, in particular in $\widetilde{G}$.
\end{proof}

\subsection{Constructable Semiwalks}

Recall the example illustrated in Figure \ref{fig:x3idems} -- the rightmost graph is the idempath identified descendant of the leftmost tree. This example illustrates that not all pairs of $\sim$-identified vertices have directed paths between; idempath identifications may generate new possible identifications. In order to classify which vertices of a graph $G$ are identified in $\widetilde{G}$, we develop the notion of \textit{$n$-constructable semiwalks}.

Let $G$ be an $X$-graph. A \textit{semiwalk} $\mathfrak{s}$ in $G$ is a sequence of edges $e_1,e_2,\dots,e_k$ with which each edge is assigned an orientation called \textit{positive} or \textit{negative}, such that $e_1e_2\dots e_k$ is a sequence of consecutively incident edges in the underlying, undirected graph of $G$ and the orientation of an edge refers to if it is traversed according to its direction in $G$, or in the reverse of its direction in $G$ via the induced path on the undirected graph. We say the semiwalk has \textit{start vertex} $\alpha(\mathfrak{s})$ and \textit{end vertex} $\omega(\mathfrak{s})$ given by the start and end of this path respectively. Semiwalks may be thought of as paths in $G$ where direction is ignored. We introduce a set of formal symbols \[X^{-1} = \left\{x^{-1}\ |\ x \in X \right\}\] which are disjoint from and in bijection with $X$. We will frequently consider words $w \in {\left(X \cup X^{-1} \right)}^*$, which we call \textit{positive} if $w \in X^*$ and \textit{negative} if $w \in \left(X^{-1}\right)^*$. The \textit{inverse} of a word $w = x_1\cdots x_k \in {\left(X \cup X^{-1} \right)}^* $ is $w^{-1} := {x_k}^{-1}\cdots{x_1}^{-1}$ with the convention that $\left(x^{-1} \right)^{-1} = x$ for all $x \in X$.

A semiwalk $\mathfrak{s}$ with edges $e_1,e_2,\dots,e_k$ in $G$ has \textit{label} $x_1\cdots x_k \in {\left(X \cup X^{-1}\right)}^*$ where \[x_i = \left\{
        \begin{array}{ll}
          \lambda(e_i)&\textrm{ if } e_i \textrm{ is traversed positively by }\mathfrak{s} \\
          \lambda(e_i)^{-1}&\textrm{ if } e_i \textrm{ is traversed negatively by }\mathfrak{s}
        \end{array}
      \right.\]

\begin{example}\label{ex:semiwalk}
    Let $X = \left\{ x,y \right\}$. The graph in Figure \ref{fig:semiwalk} has a semiwalk from $u$ to $v$ with label $y^{-1}x^{-1}x^{-1}y^{-1}yxyxyyxx$. The label of this semiwalk is neither positive nor negative.
    \begin{figure}[ht]
        \centering
        \begin{tikzpicture}
            \GraphInit[vstyle=Empty]
            \SetVertexSimple[MinSize = 1pt]
            \SetUpEdge[lw = 0.5pt]
            \tikzset{EdgeStyle/.style={->-}}
            \tikzset{VertexStyle/.append style = {minimum size = 3pt, inner sep = 0pt}}
            \SetVertexNoLabel
            \SetGraphUnit{2}
            
            \node (A) at ( 0,0) {\large$+$};
            \Vertex[x=-2,y=0]{C1}
            \node (C2) at ( -2,2) {\large$\times$};
            \Vertex[x=0,y=2]{C3}

            \Vertex[x=2,y=2]{B1}
            \Vertex[x=4,y=2]{B2}
            \Vertex[x=6,y=2]{B3}
            \node (labu) at ( 6.5,2) {\large$u$};

            \Vertex[x=2,y=0]{D1}
            \Vertex[x=4,y=0]{D2}
            \Vertex[x=6,y=0]{D3}
            
            \node (labv) at ( 6.5,0) {\large$v$};
            
            \Edge(A)(C1)\draw (A) -- (C1) node [midway, below=2pt] {$x$};
            \Edge(C1)(C2)\draw (C1) -- (C2) node [midway, left=2pt] {$y$};
            \Edge(C2)(C3)\draw (C2) -- (C3) node [midway, above=2pt] {$x$};
            \Edge(C3)(A)\draw (C3) -- (A) node [midway, left=2pt] {$y$};

            \Edge(A)(B1)\draw (A) -- (B1) node [midway, left=2pt,above=2pt] {$x$};
            \Edge(B1)(B2)\draw (B1) -- (B2) node [midway, above=2pt] {$x$};
            \Edge(B2)(B3)\draw (B2) -- (B3) node [midway, above=2pt] {$y$};

            \Edge(A)(D1)\draw (A) -- (D1) node [midway, below=2pt] {$y$};
            \Edge(D1)(D2)\draw (D1) -- (D2) node [midway, below=2pt] {$x$};
            \Edge(D2)(D3)\draw (D2) -- (D3) node [midway, below=2pt] {$x$};
            
        \end{tikzpicture}
        \caption{An $X$-graph with $X = \left\{ x,y \right\}.$}
        \label{fig:semiwalk}
    \end{figure}
\end{example}

For an $X$-generated right cancellative monoid $C$, recall that an idempath identification on $G$ is the process of identifying the endpoints of an idempath in $G$. Any idempath may be traversed by two semiwalks on $G$ in opposite directions, with labels $w$ and $w^{-1}$ (depending on the direction of traversal) where $[w]_C = 1$. Our goal for the remainder of this section is to describe which pairs of vertices in a graph are identified, using the notion of semiwalks.
  
Let $\ell$ be a word over $X \cup X^{-1}$. We say that $\ell$ is \textit{$n$-constructable} (\textit{with respect to $C$}) if $\ell$ can be obtained from the empty word by successively inserting $n$ words $w_i$ or ${w_i}^{-1}$ where $w_i \in X^*$ and $[w_i]_C = 1$, from the empty word. We say a semiwalk on a graph is \textit{$n$-constructable} if its label is $n$-constructable. We call a semiwalk or label \textit{constructable} if it is $n$-constructable for some $n$.

The only $0$-constructable word is the empty word $\epsilon$. The $1$-constructable words are exactly those representing idempotents of $M$ (or their formal inverses). This notion allows us to classify vertices which will be identified under the process of idempath identification.

\begin{thm}\label{thm:whenidentify}
    Fix an $X$-generated right cancellative monoid $C$. Two vertices $u$ and $v$ of an $X$-graph $G$ are identified in $\widetilde{G}$ if and only if there exists a constructable semiwalk between $u$ and $v$ in $G$.
\end{thm}

\begin{proof}
    First suppose there exists an $n$-constructable semiwalk between $u$ and $v$ in $G$; suppose it has label $\ell$. If the word is empty then $u=v$ so $u$
    and $v$ are trivially identified, so suppose it is non-empty. Consider the final word $w \in X^* \cup \left(X^{-1}\right)^*$ inserted in an $n$-construction of $\ell$. We may write $\ell = \ell_1w\ell_2$ for some words $\ell_1,\ell_2 \in \left(X \cup X^{-1}\right)^*$. As $w$ is a subword of $\ell$, it corresponds to an idempath $e$ in $G$. Hence there is a valid idempath identification $(w_1,w_2)$ on $G$ where $w_1,w_2 \in \{\alpha(e),\omega(e)\}$. Moreover, the word $\ell_1\ell_2$, which is $(n-1)$-constructable, is the label of a semiwalk in $(w_1,w_2)G$ between the natural images of $u$ and $v$. If $n \geq 2$, this descendant may not be idempath identified. But we may now repeat this argument with the last idempotent inserted in an $(n-1)$-construction of $\ell_1\ell_2$ to find a deeper descendant of $G$ in which there is an $(n-2)$-constructable semiwalk between the images of $u$ and $v$, and so on. This process will eventually terminate (as $n$ is finite) and thus it follows that there is some descendant $G'$ of $G$ in which there is a $0$-constructable semiwalk (i.e. the empty path) between the images of $u$ and $v$. Thus $u$ and $v$ are identified in all descendants of $G'$, in particular in $\widetilde{G'} = \widetilde{G}$ as required.

    For the converse, we utilise the following lemma.
    \begin{lem}\label{lem:semiwalkconstr}
        Let $x,y,u,v$ be vertices of $G$. Denote the natural images of $x,y,u,v$ in any descendant $K$ of $G$ by $x_K, y_K, u_K$ and $v_K$ respectively. Let $H$ be a descendant of $G$ and suppose $(x_H,y_H)$ is a well-defined idempath identification on $H$. Suppose $\ell$ is a label of an $m$-constructable semiwalk from $u_{(x,y)H}$ to $v_{(x,y)H}$ in $(x,y)H$ for some $m \in \mathbb{N}_0$. Then there exists an $n$-constructable semiwalk from $u_H$ to $v_H$ in $H$ for some $n\in\mathbb{N}_0$.
    \end{lem}
    \begin{proof}
        Suppose the semiwalk labelled $\ell$ traverses the edges $p_1,p_2,\dots p_k$ in $(x,y)H$. Consider the pre-image of these $p_i$ in $H$. For notation, we ignore the direction of each $p_i$, and consider $\alpha(p_{i+1}) = \omega(p_i)$ to be the vertex of $(x,y)H$ traversed by the semiwalk between $p_i$ and $p_{i+1}$.

        For each $1 \leq i \leq k-1$, if $\omega(p_i)_H$ and $\alpha(p_{i+1})_H$ are identified in $H$, then define $w_i = \epsilon$. Otherwise, $\omega(p_i)_H$ and $\alpha(p_{i+1})_H$ are not identified in $H$. Since they are identified in $(x,y)H$, the idempath identification $(x_H,y_H)$ identified $\omega(p_i)_H$ with $\alpha(p_{i+1})_H$. Hence $(x_H,y_H) = (\omega(p_i)_H, \alpha(p_{i+1})_H)$ or $(x_H,y_H) = (\alpha(p_{i+1})_H,\omega(p_i)_H)$ and thus there is an idempath between $\omega(p_i)_H$ and $\alpha(p_{i+1})_H$ in $H$. If this idempath is oriented towards $\alpha(p_{i+1})_H$, take $w_i$ to be its label, otherwise take $w_i$ to be the inverse of its label. Perform a similar construction with $u_H$ and $\alpha(p_1)_H$ to construct a word $w_0$, and with $\omega(p_k)_H$ and $v_H$ to construct $w_k$.
        
        Write $\ell = a_1\cdots a_k \in (X \cup X^{-1})^*$, such that each $a_i$ is $\lambda(p_i)$ or $\lambda(p_i)^{-1}$ (depending on direction traversed). By construction, we have a valid connected semiwalk in $H$ from $u_H$ to $v_H$ with label $\ell' := w_0a_1w_1a_2w_2\cdots a_kw_k$. Since each $w_i$ labels an idempath, $\ell'$ is $(m+k+1)$-constructable via $k+1$ inserts into $\ell$.
    \end{proof}

    We now return to the proof of Theorem~\ref{thm:whenidentify}. Suppose $u,v$ are identified in $\widetilde{G}$. Then there is some descendant $G' = (x_m,y_m)\dots(x_1,y_1)G$ in which $u_{G'}$ and $v_{G'}$ are identified, that is there is an (empty) $0$-constructable semiwalk from $u_{G'}$ to $v_{G'}$ in $G'$. By $m$ successive applications of Lemma \ref{lem:semiwalkconstr}, we may obtain a constructable semiwalk from $u$ to $v$ in $G$.
\end{proof}

\subsection{Semiwalks on Trees and Two-sided Cancellativity} We will often be interested in how constructable semiwalks interact with trees. Let $C$ be an $X$-generated right cancellative monoid and let $T$ be an $X$-tree. In the case where $C$ is two-sided cancellative, we may observe properties of how constructable semiwalks interact with the $C$-value of co-terminal paths.

\begin{lem}\label{lem:pathsintilde}
Let $C$ be an $X$-generated two-sided cancellative monoid and $T$ be an $X$-tree.
  Then every retract of every descendant of $T$ has the properties that any cycle is an idempath and if $w,w'$ label co-terminal paths between distinct vertices $u$ to $v$ in $G$, then $[w]_C = [w']_C$. In particular, $\widetilde{T}$ has these properties.
\end{lem}

\begin{proof}
    Since retracts are birooted subgraphs and the claimed properties are clearly inherited by subgraphs, it clearly suffices to prove the claim for the case that $G$ is a descendant of $T$.

    Write $G = (x_m,y_m)\dots(x_1,y_1)T$. We perform induction on $m$ -- if $m = 0$, the statement is true vacuously. Suppose then that the statement holds true in some $G_{k} := (x_{k},y_{k})\dots(x_1,y_1)T$. We show the properties hold in $G_{k+1} := (x_{k+1},y_{k+1})G_{k}$.
    
    We first show that any cycle in $G_{k+1}$ is an idempath, that is it has label equalling $1$ in $C$. It is sufficient to show the result for simple cycles, since any non-simple cycle may be obtained by successively composing and inserting simple cycles into simple cycles. Thus showing each insert has label equalling $1$ will ensure the label of the full cycle also equals $1$. 
    
    Suppose then that $z_0 \xrightarrow{p_1}z_1 \xrightarrow{p_2}\dots \xrightarrow{p_{n-1}} z_{n-1}\xrightarrow{p_n}z_0$ is a simple cycle in $G_{k+1}$ for vertices $z_i$ in $G_{k+1}$ and labels $p_i \in X$. If this cycle lifts to some cycle in $G_k$, then $[p_1p_2\cdots p_n]_C = 1$ by assumption. Otherwise, this cycle must have been created by the identification of
    $x_{k+1}$ and $y_{k+1}$. It follows that $z_i = x_{k+1} = y_{k+1}$ in $G_{k+1}$ for some $z_i$. We may relabel our cycle to ensure $z_0 = x_{k+1} = y_{k+1}$ without loss of generality. It follows that the edges of our simple cycle lifts to $G_k$ to either a path \begin{equation}\label{eq:path1}
        x_{k+1} \xrightarrow{p_{1}}z_1'\xrightarrow{p_2}\dots \xrightarrow{p_{n-1}} z_{n-1}'\xrightarrow{p_n}y_{k+1}
    \end{equation}or a path\begin{equation}\label{eq:path2}
       y_{k+1} \xrightarrow{p_{1}}z_1'\xrightarrow{p_2}\dots \xrightarrow{p_{n-1}} z_{n-1}'\xrightarrow{p_n}x_{k+1}
    \end{equation}where the vertex $z_j$ lifts to some vertex $z_j'$ in $G_k$.
    
    If the cycle lifts to the path \eqref{eq:path1}, then since the identification $(x_{k+1},y_{k+1})$ is defined on $G_k$, there exists some idempath from $x_{k+1}$ to $y_{k+1}$ in $G_k$. By our inductive assumption, all paths from $x_{k+1}$ to $y_{k+1}$ are labelled by words defining equal elements of $C$. Hence the path \eqref{eq:path1} must be an idempath and have $[p_1p_2\cdots p_n]_C = 1$. Otherwise, if the cycle lifts to the path \eqref{eq:path2}, then as there exists an idempath from $x_{k+1}$ to $y_{k+1}$, say labelled $z$, there exists a cycle in $G_k$ labelled $zp_1\cdots p_n$. By inductive assumption, $1 = [zp_1\cdots p_n]_C = [1p_1\cdots p_n]_C = [p_1\cdots p_n]_C$ as required.

    We now show the result for co-terminal paths. Suppose $u,v$ are distinct vertices in $G_{k+1}$ with two paths labelled $u \xrightarrow{w} v$ and $u \xrightarrow{w'} v$ in $G_{k+1}$. Our goal is to show that $[w]_C = [w']_C$. We first show our result when the paths are vertex disjoint and simple. As $u$ and $v$ are distinct vertices in $G_{k+1}$, certainly their respective lifts are still disjoint in $G_k$. If the paths both lift to connected co-terminal paths, then the result follows immediately. Otherwise, the identification of $x_{k+1}$ and $y_{k+1}$ must have created these paths in $G_{k+1}$. Since the paths are vertex disjoint in $G_{k+1}$, and the identification identified exactly two vertices, at least one of the paths must have lifted to a connected path from some pre-image of $u$ to some pre-image of $v$ -- without loss of generality suppose the path labelled $w$ did so. Call these pre-images $u'$ and $v'$ respectively. It follows that the paths lifted in such a way that in $G_{k}$, we observe one of the situations shown in Figure \ref{fig:gk}, where $w' \equiv w_1w_2$ for some words $w_1,w_2 \in X^*$ and an idempath labelled $z \in X^*$ from $x_{k+1}$ to $y_{k+1}$.

    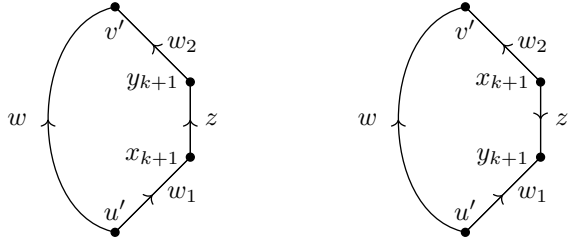
\begin{figure}[ht]
        \centering
        {
        \begin{tikzpicture}
            \GraphInit[vstyle=Empty]
            \SetVertexSimple[MinSize = 1pt]
            \SetUpEdge[lw = 0.5pt]
            \tikzset{EdgeStyle/.style={->-}}
            \tikzset{VertexStyle/.append style = {minimum size = 3pt, inner sep = 0pt}}
            \SetVertexNoLabel
            \SetGraphUnit{2}
            
            \Vertex[x=0,y=0]{u}\node (l) at ( 0,0.3) {$u'$};
            \Vertex[x=0,y=3]{v}\node (l) at ( 0,2.7) {$v'$};
            
            \Vertex[x=1,y=1]{elow}\node (l) at ( 0.5,1) {$x_{k+1}$};
            \Vertex[x=1,y=2]{ehigh}\node (l) at ( 0.5,2) {$y_{k+1}$};
                      
            \Edge[style={bend left = 75}](u)(v)\node (l) at ( -1.3,1.5) {$w$};
            \Edge(u)(elow)\draw (u) -- (elow) node [midway, right=2pt] {$w_1$};
            \Edge(elow)(ehigh)\draw (elow) -- (ehigh) node [midway, right=2pt] {$z$};
            \Edge(ehigh)(v)\draw (ehigh) -- (v) node [midway, right=2pt] {$w_2$};
            
        \end{tikzpicture}
        }\hspace{0.1\textwidth}
        {
        \begin{tikzpicture}
            \GraphInit[vstyle=Empty]
            \SetVertexSimple[MinSize = 1pt]
            \SetUpEdge[lw = 0.5pt]
            \tikzset{EdgeStyle/.style={->-}}
            \tikzset{VertexStyle/.append style = {minimum size = 3pt, inner sep = 0pt}}
            \SetVertexNoLabel
            \SetGraphUnit{2}
            
            \Vertex[x=0,y=0]{u}\node (l) at ( 0,0.3) {$u'$};
            \Vertex[x=0,y=3]{v}\node (l) at ( 0,2.7) {$v'$};
            
            \Vertex[x=1,y=1]{elow}\node (l) at ( 0.5,1) {$y_{k+1}$};
            \Vertex[x=1,y=2]{ehigh}\node (l) at ( 0.5,2) {$x_{k+1}$};
                      
            \Edge[style={bend left = 75}](u)(v)\node (l) at ( -1.3,1.5) {$w$};
            \Edge(u)(elow)\draw (u) -- (elow) node [midway, right=2pt] {$w_1$};
            \Edge(ehigh)(elow)\draw (elow) -- (ehigh) node [midway, right=2pt] {$z$};
            \Edge(ehigh)(v)\draw (ehigh) -- (v) node [midway, right=2pt] {$w_2$};
            
        \end{tikzpicture}}
        \caption{The lifts of the coterminal paths labelled $w$ and $w' \equiv w_1w_2$.}
        \label{fig:gk}
    \end{figure}
    If the paths lift in such a way that the left graph of Figure \ref{fig:gk} arises, then by inductive assumption, we have $[w]_C = [w_1zw_2]_C = [w_1\cdot1\cdot w_2]_C = [w_1w_2]_C \equiv [w']_C$ and we are done.

    Suppose instead the paths lift in such a way that the right graph of Figure \ref{fig:gk} arises. Certainly, there exists some path from the start vertex of $G_k$ to $x_{k+1}$, say labelled $p$, and some path from the start vertex of $G_k$ to $u'$, say labelled $q$. From our inductive assumption, it follows that $[p]_C = [pz]_C = [qw_1]_C$ and $[pw_2]_C = [qw]_C$. Thus $[qw_1w_2]_C = [qw]_C$ and so by (left) cancellativity, $[w]_C = [w_1w_2]_C \equiv [w']_C$ as required.

    Now note that if our paths are not simple, then they may be obtained by inserting cycles into a simple path, which by the result for cycles above will maintain the $C$-value of their labels. Moreover, if our paths are not vertex disjoint, then we use an inductive argument on the number of times the path labelled $w'$ intersects the path labelled $w$ (the base case of $0$ intersections given above). Suppose $u \xrightarrow{w'} v$ intersects $u \xrightarrow{w} v$ $K$ times. Consider the shortest, non-trivial prefix $a$ of $w'$ which labels a subpath of $u \xrightarrow{w'}$ to a vertex $r$ also on $u \xrightarrow{u} v$. Write $w' \equiv ab$ and $w \equiv cd$ where our paths split as $u \xrightarrow{a} r \xrightarrow{b} v$ and $u \xrightarrow{c} r \xrightarrow{d} v$. By choice of $a$, the paths $u \xrightarrow{a} r$ and $u \xrightarrow{c} r$ are vertex disjoint, and the path $r \xrightarrow{b} v$ intersects $r \xrightarrow{d} v$ at most $K-1$ times. Thus by inductive assumption, we have $[a]_C = [c]_C$ and $[b]_C = [d]_C$. Hence $[w']_C = [ab]_C = [cd]_C = [w]_C$ as required.

    Returning to our main induction, the result now follows.
\end{proof}

Suppose $G$ is an $X$-graph which is a retract of a descendant of some $X$-tree $T$. Given any vertex $v$ of $G$, certainly there exists a directed path from the
root to $v$ and by Lemma~\ref{lem:pathsintilde} every such path has a label which evaluates to the same element of $C$. Thus we may define the \textit{$C$-value} of the vertex $v$, denoted $[v]_C$, to be the value in $C$ of the label of any path from the root to $v$.

\begin{cor}\label{cor:tildekeepsCvalue}
    Let $C$ be two-sided cancellative and let $G$ be an $X$-graph which is a retract of a descendant of some $X$-tree. 
    If two vertices of $G$ are identified in any retract of any descendant of $G$, then they have the same $C$-value. Moreover, if there exists a directed path from $u$ to $v$ in any retract of any descendant of $G$, then it is an idempath.
\end{cor}

\begin{proof}
    Let $u$ and $v$ be vertices identified in some retract $H$ of a descendant of $G$. Let $p, q \in X^*$ be the labels of some paths in $G$ from the root to $u$ and $v$ respectively, so that $[u]_C = [p]_C$ and $[v]_C = [q]_C$. Then $p$ and $q$ label coterminal paths in $H$, so by Lemma~\ref{lem:pathsintilde} must have the same value in $C$.
    
Moreover, if $K$ is a retract of a descendant of $G$ in which there exists a directed path from $u$ to $v$, say with label $z$, then $[u]_C [z]_C = [v]_C$ or $[v]_C[z]_C = [u]_C$. Since $[u]_C = [v]_C$, we have $[z]_C = 1$ by (left) cancellativity.
\end{proof}

\begin{lem}\label{lem:standard}
    Let $T$ be an $X$-tree and let $C$ be two-sided cancellative. Suppose $u$ and $v$ are distinct vertices of $T$ with a constructable semiwalk labelled $\ell$ on $T$ from $u$ to $v$. Let $t$ be the unique vertex of $T$, closest to the root, traversed by the semiwalk. Then there exists a constructable semiwalk between $u$ and $v$ in which the last idempotent inserted to construct its label corresponds to an idempath beginning at $t$.
\end{lem}

\begin{proof}
    If $\ell$ is $1$-constructable, then $t \in \{u,v\}$ and either $\ell$ or $\ell^{-1}$ labels a positively-labelled idempath beginning at $t$ as required. Now suppose $\ell$ is not $1$-constructable. Note that there exist paths in $T$ from $t$ to $u$ and $t$ to $v$. If $t \in \{u,v\}$, then $u$ and $v$ are connected in $T$. Thus by Corollary \ref{cor:tildekeepsCvalue} this path is an idempath -- the positive orientation of this path gives us our desired semiwalk. Now suppose further that $t \notin \{u,v\}$.
                
    Write $\ell \equiv pxyq$ where $p,q \in \left\{X\cup X^{-1}\right\}^*$ and $x,y \in X \cup X^{-1}$ such that $px$ is the first prefix of $\ell$ tracing a semiwalk ending at $t$. Since $t$ is the `lowest' vertex visited, it follows that $x \in X^{-1}$ and $y \in X$. Thus, $x$ and $y$ were letters from different words inserted in any $n$-construction of $\ell$. Now fix any $n$-construction. Let $e_x$ be the inserted word containing $x$, and $e_y$ be the inserted word containing $y$. We split into cases depending on the order in which $e_x$ and $e_y$ were inserted.
    
    \textbf{\textit{Case 1: $e_y$ was inserted after $e_x$.}}\\
    Consider the smallest subword of $\ell$ containing every letter of the word inserted containing $y$; call this subword $\ell'$. Since $e_y$ was inserted after $e_x$, $\ell'$ consists of letters from $e_y$ and letters from words inserted after $e_y$. Moreover, if an inserted word has any letter in $\ell'$, then all of its letters are in $\ell'$. It follows that $\ell'$ is itself $K$-constructable for some $K < n$. Since $\ell'$ does not include the letter $x$, we may rewrite $\ell$ as $\ell = px\ell' r$ for some $r \in \left\{X\cup X^{-1}\right\}^*$. Since $e_y$ was inserted after $e_x$, we may find an insertion order to further see that the word $pxr$ is $(n-K)$-constructable (by inserting all of the words with letters in $\ell'$ last).
                
    By considering the semiwalk traced by $\ell'$, we see that it labels a semiwalk beginning at $t$ completely in $\mathrm{Cone}_T(t)$. Let $w$ be the label of the unique path from $t$ to $\omega(\ell')$. Since $\ell'$ is $K$-constructable, we must have that $t$ is identified with the $\omega(\ell')$ in $\widetilde{T}$; hence $w$ labels an idempath by Corollary \ref{cor:tildekeepsCvalue}. Thus the word $pxwr$ is $(n-K+1)$-constructable. Moreover, it labels a semiwalk on $T$ from $u$ to $v$ as required.
    
    \textbf{\textit{Case 2: $e_y$ was inserted before $e_x$.}}\\
    Consider the reverse semiwalk labelled $\ell^{-1} = q^{-1}y^{-1}x^{-1}p^{-1}$. By inverting our construction of $\ell$, we have an $n$-construction of $\ell^{-1}$ in which the word containing $x^{-1}$ was inserted after the word containing $y^{-1}$. We may therefore use a dual argument to that in Case 1 (considering instead the final time the vertex $t$ is traversed) to determine that there is a constructable semiwalk from $v$ to $u$ on $T$ in which the final inserted word traverses an idempath beginning at $t$. Taking the inverse label of this semiwalk is our desired semiwalk.
\end{proof}
                    
Lemma \ref{lem:standard} allows us to define the following. We call an $n$-constructable semiwalk from $u$ to $v$ on $T$ \textit{standard} if there is an $n$-construction of the semiwalk in which the final inserted word corresponds to an idempath in $T$ starting at the lowest vertex of $T$ traversed by the semiwalk. Theorem \ref{thm:whenidentify} and Lemma \ref{lem:standard} ensure that such a semiwalk exists for any pair of identified vertices $u$ and $v$ when $T$ is a tree and $C$ is two-sided cancellative.

Given an $X$-graph $\Gamma$, we define its \textit{condensation} $\textrm{Cond}(\Gamma)$ to be the $X$-graph obtained by contracting all vertices of $\Gamma$ which lie in the same strongly connected component into a single vertex. Formally, the vertex set $\textrm{Cond}(\Gamma)$ is $\{\ol v \colon v \in V(\Gamma)\}$ with start vertex $\ol {\alpha(\Gamma)}$ and end vertex $\ol{\omega(\Gamma)}$, and edge set $$\{ \ol u \xrightarrow{x} \ol v \colon  u \xrightarrow{x} v \hbox{ is an edge of }\Gamma,\ol u \neq \ol v\}.$$

An $X$-graph is called a \textit{directed acyclic graph} if it contains no directed cycle with at least $1$ edge, or equivalently, if its strongly connected components
are single vertices. In particular, all $X$-trees are directed acyclic graphs, and for any directed $X$-graph $\Gamma$, $\textrm{Cond}(\Gamma)$ is a directed
acyclic graph.

\begin{prop}\label{prop:pretzelcondistree}
For any $X$-tree $T$, $\mathrm{Cond}(\widetilde{T})$ is an $X$-tree.
\end{prop}

\begin{proof}
We prove the statement by induction on the number of idempath identifications used to construct $\widetilde{T}$. Suppose $\Gamma=\widetilde{T}$ and $T=\Gamma_0, \Gamma_1, \ldots, \Gamma_n=\Gamma$ is a sequence of $X$-graphs where $\Gamma_{k+1}$ is obtained from $\Gamma_k$ by an idempath identification, and suppose $\textrm{Cond}(\Gamma_k)$ is an $X$-tree. Let $v_1 \xrightarrow{x_1} v_2   \xrightarrow{x_2} \ldots  \xrightarrow{x_{m-1}} v_m$ be the idempath in $\Gamma_k$ that is identified to obtain $\Gamma_{k+1}$.
For each edge $v_i  \xrightarrow{x_{i}} v_{i+1}$, either $\ol v_i =\ol v_{i+1}$ or $\ol v_i  \xrightarrow{x_{i}} \ol v_{i+1}$ is an edge in $\textrm{Cond}(\Gamma_k)$, so
there is a (possibly empty) path $\ol v_{j_1} \rightarrow \ol v_{j_2} \rightarrow \ldots \rightarrow \ol v_{j_t}$ in the $X$-tree $\textrm{Cond}(\Gamma_k)$ such
that every $v_i$ is contained in $\ol v_{j_k}$ for some $k$. 

Denote the graph morphism $\Gamma_k \to \Gamma_{k+1}$ by $f$. Then in $\Gamma_{k+1}$, the vertices $f(v_1), f(v_2), \ldots, f(v_m)$ lie on a directed cycle and are therefore all in the same strongly connected component. We furthermore claim that if $u$ is a vertex of $\Gamma_k$ such that $\ol u \neq \ol v_i$ for any $i$, then $\ol {f(u)} \neq \ol {f(v_1)}=\ldots =\ol {f(v_m)}$. Indeed, to show the contrapositive, assume that there is a directed path $p$ from $f(u) \to f(v_i)$ and a directed path $q$ from $f(v_j) \to f(u)$. Note that without loss of generality, we may assume that $p$ and $q$ have no internal vertices in the set $\{f(v_1), \ldots, f(v_m)\}$, and therefore they must lift to paths $u \to v_i$ and $v_j \to u$ respectively. Then either $\ol u = \ol v_i$ or $\ol u= \ol v_j$, or we have a directed path $\ol v_j \rightarrow \ol u \rightarrow \ol v_{i}$ in $\textrm{Cond}(\Gamma_k)$, but since $\textrm{Cond}(\Gamma_k)$ is an $X$-tree, by assumption $\ol u$ must then be one of the vertices $\{\ol v_{j+1}, \ldots, \ol v_{i-1}\}$, which proves the statement.

It thus follows that the graph $\textrm{Cond}(\Gamma_{k+1})$ is obtained from the graph $\textrm{Cond}(\Gamma_{k})$ by contracting the vertices in the path  $\ol v_{j_1} \rightarrow \ol v_{j_2} \rightarrow \ldots \rightarrow \ol v_{j_t}$, and is therefore an $X$-tree. The statement follows by induction.
\end{proof}

\begin{prop}\label{prop:cayleygraphchunks}
Let $C$ be an $X$-generated two-sided cancellative monoid and let $\Gamma$ be an $X$-graph which is a retract of a descendant of an $X$-tree. For any vertex $v$, denote by $\Gamma_v$ the subgraph of $\Gamma$ spanned by all directed paths from $\alpha(\Gamma)$ to $v$. Then there is an injective morphism of edge-labeled directed graphs from $\Gamma_v$ into the Cayley graph of $C$.
\end{prop}

\begin{proof}
Denote the Cayley graph of $C$ by $\textrm{Cay}(C;X)$. Recall that this is an $X$-labelled directed graph with vertex set $C$ and with edges of the form $c \xrightarrow{x} {cx}$ for each $c \in C$ and $x \in X$.

Define a map $\phi: V(\Gamma) \to V(\textrm{Cay}(C;X))$ by $\phi(u) = [u]_C$. To show that this defines a morphism of edge-labelled directed graphs,
we need to show that if there is an edge $u \xrightarrow{x} u'$ in $\Gamma$, then there exists an edge $\phi(u) \xrightarrow{x} \phi(u')$ in $\textrm{Cay}(C;X)$. Indeed, by Lemma \ref{lem:pathsintilde}, since in $\Gamma$ (and particularly in the subgraph $\Gamma_v$) co-terminal paths have labels with equal value in $C$, we have that $[u]_C x = [u']_C$. Thus an edge $\phi(u) \xrightarrow{x} \phi(u')$ certainly exists in $\textrm{Cay}(C;X)$. This $\phi$ extends to a graph morphism.
We claim that this morphism is injective.

Suppose $u,u' \in V(\Gamma_v)$ have $[u]_C = [u']_C$. Note $\textrm{Cond}(\Gamma_v)$ is a rooted subgraph of $\textrm{Cond}(\Gamma)$ -- one sees it is a directed tree (in fact a directed path) with root vertex $\ol{\alpha(\Gamma)}$ and exactly one leaf vertex $\ol{v}$. It follows that there must exist a directed path between $u$ and $u'$ in some direction in $\Gamma_v$, and indeed in $\Gamma$. By Lemma \ref{lem:pathsintilde}, it must have label equalling $1$ in $C$ and therefore be a cycle. Thus $u$ and $u'$ are the same vertex of $\Gamma_v$, that is $\phi$ is injective on vertices.

Now suppose edges $e$ and $f'$ in $\Gamma_v$ have $\phi(e) = \phi(f)$. Then in particular we must have $\lambda(e)=\lambda(f)$, and by injectivity on the vertices, also $\alpha(e)=\alpha(f)$ and $\omega(e)=\omega(f)$. Thus the edges $e$ and $f$ are co-terminal edges in $\Gamma_v$ with the same label. In particular, there exists a retraction on $\Gamma$ which identifies these edges, so as $\Gamma$ is retracted, $e=f$ indeed. 
\end{proof}

\section{Pretzel Monoids}\label{sec:pretzel}

Our aim in this section is to introduce a kind of left adequate expansion of a given $X$-generated right cancellative monoid $C$. We do this first by a geometric
construction, and later show that the same monoid can be obtained by a natural presentation in the category of left adequate monoids.

We shall need some basic facts about right and two-sided cancellative monoids, which will allow us sometimes to restrict to the two-sided cancellative case. An $X$-generated \textit{special} right cancellative monoid is one given by a presentation of the form $\RC \langle X \mid w_i=1,\ i \in I \rangle,$ that is, where all relations have the identity on one side of the equality. When we say $X$-generated special right cancellative monoid, we mean a right cancellative monoid which is given by a special presentation with respect to the generating set $X$. Given a right cancellative monoid $C$, if we define
$$C'= \RC \langle X \mid w=1 \hbox{ for all } w\in X^* \hbox{ such that } [w]_C=1 \rangle,$$
then it is easy to see that exactly the same words represent over $X$ represent $1$ in $C$ and in $C'$.
We have the following straightforward characterisation of special right cancellative monoids:

	\begin{thm}\label{thm:specialrightiscanc}
		Special right cancellative monoids are exactly the free products of free monoids and groups (in the variety of monoids). In particular, they are
		two-sided cancellative.
	\end{thm}
	
	\begin{proof}
		Given a free product $G \ast Y^*$, where $G$ is generated as a monoid by $X$, it has a special presentation
		$\RC \langle X \cup Y \mid w=1 \hbox{ whenever } w \in X^\ast \hbox{ and } [w]_G=1 \rangle.$
		Conversely, given a special right cancellative monoid
		$C=\RC \langle X \mid w_i=1, i \in I \rangle,$
		partition $X$ as $X=Y \cup Z$, where the letters in $Y$ do not appear in any relator, and the letters in $Z$ appear in at least one. We claim that every letter $z \in Z$ represents an invertible element in $C$. Assume $uzv$ is a relator, that is, $[uzv]_C=1$. Then $[uzvu]_C=[u]_C$ and so $[zvu]_C=1$ by right cancellativity, and similarly, $[vuzv]_C=[v]_C$ and hence $[vuz]_C=1$ by right cancellativity. It follows that $[z]_C$ is in the group of units of $C$ with inverse $[vu]_C$. In particular, $Z$ generates a subgroup. Since letters in $Y$ do not appear in the relators, $C$ splits as a free product of $\langle Z \rangle \ast Y^*$ with $\langle Z \rangle$ a group, as we needed.
	\end{proof}
    
    Combining with our observation above, and noting that free products of groups and free monoids are two-sided cancellative, we have:
    \begin{cor}\label{cor:twosided}
        For every $X$-generated right cancellative monoid $C$, there is an $X$-generated two-sided cancellative monoid $C'$ in which exactly the same words
        represent the identity element.
    \end{cor}

Now fix an $X$-generated right cancellative monoid $C$ (which at certain points we may assume further to be two-sided cancellative). In the remainder of this section, all idempath identifications are done with respect to $C$ (or equivalently with respect to $C'$, see Corollary \ref{cor:twosided}).

We define a binary operation, called \textit{unpruned multiplication} (or `\textit{gluing}') on the set of all $X$-graphs as follows. For $X$-graphs $G$ and $H$, the \textit{unpruned product} $G \times H$ is given by the graph $G \cup H$ with the end vertex of $G$ identified with the start vertex of $H$, with start vertex given by the start vertex of the natural embedded copy of $G$, and end vertex given by the end vertex of the natural embedded copy of $H$. Note that $G \times H$ is indeed itself an $X$-graph.

We additionally define a unary operation, denoted $(+)$ as follows. For an $X$-graph $G$, we let $G^{(+)}$ denote the graph $G$ but with start and end vertex located at the start vertex of $G$. As before, note again that $G^{(+)}$ is itself an $X$-graph.

It is clear that $\times$ is an associative operation on the set of $X$-graphs with a two-sided identity given by the single-vertex, no-edge graph. Moreover, $(+)$ is an idempotent operation on the set of $X$-graphs and the subsemigroup generated by the image of $(+)$ is commutative under $\times$. These operations generalise those given for $X$-trees in \cite{kambites:free,kambites:freeleft}.

We now recall the concept of retraction onto the core of an $X$-graph $G$ \cite{hell:core}, and the notation given by $\overline{G}$. Together with our process of idempath identification, we now look at how our operations interact with the unpruned multiplication.

\begin{lem}\label{lem:timeswithtildeandprune}
Let $G,H$ be $X$-graphs. Then $\overline{\overline{G}\times\overline{H}} = \overline{G \times H}$ and $\widetilde{\widetilde{G} \times \widetilde{H}} = \widetilde{G \times H}$.
\end{lem}

\begin{proof}
First, suppose $\rho_G$ and $\rho_H$ are retracts on $G$ and $H$ which have images $\overline{G}$ and $\overline{H}$ respectively. Since the amalgamated vertex in $G \times H$ is a distinguished vertex of both $G$ and $H$, it must be fixed by both $\rho_G$ and $\rho_H$. It follows that there is a unique morphism $\rho$ on $G \times H$ which extends the maps $\rho_G$ and $\rho_H$ in the expected way. Since $\rho_G$ and $\rho_H$ are idempotent, it follows that $\rho$ is idempotent. Clearly $\rho$ fixes the start and end vertex of $G \times H$, hence $\rho$ is a retract on $G \times H$. It has image $\rho(G \times H) = \overline{G} \times \overline{H}$. By confluence of retractions (see, for example, \cite{hell:core}), we have that $\overline{G \times H} = \overline{ \rho(G \times H) } = \overline{\overline{G}\times\overline{H}}$.

Next, observe that any idempath in $G$ is also an idempath in $G \times H$. Similarly, any idempath in $H$ is an idempath in $G \times H$. Thus we have that $\widetilde{G} \times \widetilde{H} \in D(G \times H)$. Clearly then $\widetilde{\widetilde{G} \times \widetilde{H}} = \widetilde{G \times H}$.
\end{proof}

\begin{lem}\label{lem:pluswithtildeandprune}
Let $G$ be an $X$-graph. Then $\overline{G^{(+)}} = \overline{{\overline{G}}^{(+)}}$ and $\widetilde{G^{(+)}} = \widetilde{G}^{(+)}$.
\end{lem}

\begin{proof}
Suppose $\rho$ is a retract of $G$ with image $\overline{G}$. By definition, $\rho$ defines a graph endomorphism on $G$, with $V(G) = V\left(G^{(+)}\right)$ and $E(G) = E\left(G^{(+)}\right)$. Also, $\rho$ fixes the start vertex of $G$, so certainly fixes the start and end vertex of $G^{(+)}$. Thus $\rho$ defines a retract of $G^{(+)}$ with image $\overline{G}^{(+)}$. By confluence of retractions, we have that $\overline{G^{(+)}} = \overline{\rho(G^{(+)})} = \overline{{\overline{G}}^{(+)}}$.

Moreover, it is clear that idempath identifications are independent of the location of the end vertex of $G$. Hence clearly $\widetilde{G^{(+)}} = \widetilde{G}^{(+)}$ since $G^{(+)}$ and $G$ have the same underlying graph.
\end{proof}

\begin{lem}\label{lem:stacking}
Let $G$ be an $X$-graph. Then $\widetilde{\overline{\widetilde{G}}} = \overline{\widetilde{G}} = \overline{\widetilde{\overline{G}}}$.
\end{lem}

\begin{proof}
We first show the first equality, that is the pruning of a idempath identified graph is idempath identified. Suppose there exists vertices $a,b \in V(G)$ and an idempath from $a$ to $b$ in $\overline{\widetilde{G}}$. Note that the vertices $a,b$ and this idempath must exist in $\widetilde{G}$ since $\overline{\widetilde{G}}$ is a subgraph of $\widetilde{G}$. But $\widetilde{G}$ is idempath identified, so $a \equiv b$ in $\widetilde{G}$. Then certainly $a \equiv b$ in $\overline{\widetilde{G}}$. Hence $\overline{\widetilde{G}}$ is idempath identified and thus $\overline{\widetilde{G}} = \widetilde{\overline{\widetilde{G}}}$ as required.

We now show the second equality. We first show that there exists graph morphisms $\phi: \widetilde{\overline{G}} \hookrightarrow \widetilde{G}$ and $\psi: \widetilde{G} \twoheadrightarrow \widetilde{\overline{G}}$. Since $\overline{G}$ is a pruned subgraph of $G$, there certainly exist some embedding $\iota : \overline{G} \hookrightarrow G$ and some retraction $\rho: G \twoheadrightarrow \overline{G}$ such that $\rho \circ \iota$ is the identity map. Moreover, there exists surjections $f: \overline{G} \twoheadrightarrow \widetilde{\overline{G}}$ and $g: G \twoheadrightarrow \widetilde{G}$ given by the respective $\sim$ maps. A diagram of these maps can be seen in Figure \ref{fig:maps}.

\begin{figure}[ht]
    \centering
    \begin{tikzcd}
        \overline{G} \arrow[rr, "\iota", hook, bend right] \arrow[dd, "f"', two heads] &  & G \arrow[ll, "\rho", two heads, bend right] \arrow[dd, "g", two heads]\\
        &  & \\
        {\widetilde{\overline{G}}} \arrow[rr, "\vartheta", dotted, hook, bend right]     &  & {\widetilde{G}} \arrow[ll, "\psi", two heads, dotted, bend right]       
    \end{tikzcd}
    \caption{The maps of Lemma \ref{lem:stacking}.}
    \label{fig:maps}
\end{figure}
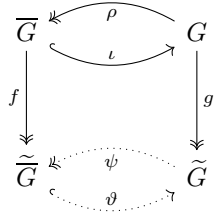

We first show that $g \circ \iota$ factors through $f$. Indeed, if two vertices $a,b$ of $\overline{G}$ are identified by $f$, then there is an constructable semiwalk between them in $\overline{G}$. As $\iota$ is an embedding, this semiwalk maps to a constructable semiwalk in $G$ from $\iota(a)$ to $\iota(b)$, hence $(g\circ \iota)(a) = (g \circ \iota)(b)$. Thus $\ker(f) \subseteq \ker(g \circ \iota)$. This gives rise to a well-defined map $\vartheta: \widetilde{\overline{G}} \to \widetilde{G}$ given on edges/vertices by $a \mapsto (g\circ\iota)(A)$ where $f(A) = a$. Moreover $\vartheta$ is a graph morphism: we note that for any edge $e$ adjacent to vertex $a$ in $\widetilde{\overline{G}}$, there exists some pre-images $E$ and $A$ under $f$ in which $A$ is adjacent to $E$ in $\overline{G}$. Since $\iota$ and $g$ are graph morphisms, it follows that $(g\circ \iota)(E)$ is adjacent to $(g\circ \iota)(A)$ in $\widetilde{G}$, and hence that $\vartheta$ is a graph morphism. It remains to show that $\vartheta$ is injective. It is sufficient to show that $\ker(g \circ \iota) \subseteq \ker(f)$. Indeed, if two vertices $a,b$ of $\overline{G}$ are identified by $g \circ \iota$, then there exists a constructable semiwalk from $\iota(a)$ to $\iota(b)$ in $G$. The image of this semiwalk under the retraction $\rho$ is a constructable semiwalk from $a$ to $b$ in $\overline{G}$, and thus $f(a) = f(b)$. Hence $\ker(g \circ \iota) = \ker(f)$ and $\vartheta$ is an embedding.

We now similarly construct $\psi$. We have that $f \circ \rho$ factors through $g$, since if two vertices $a,b$ of $G$ are identified by $g$, then the constructable semiwalk between them maps via $\rho$ to a constructable semiwalk between $\rho(a)$ and $\rho(b)$ in $\overline{G}$. Hence $\rho(a)$ and $\rho(b)$ are identified by $f$. As before, this gives rise to a well-defined map $\psi: \widetilde{G} \to \widetilde{\overline{G}}$ given on edges/vertices by $a \mapsto (f \circ \rho)(A)$ where $g(A) = a$. Via identical reasoning to above, $\psi$ may be seen to be a graph morphism. Moreover, as $\psi \circ g = f \circ \rho$, $\psi$ is surjective since $g,f$ and $\rho$ are surjective.

We have therefore completed the diagram in Figure \ref{fig:maps}. The map $\vartheta$ allows us to realise $\widetilde{\overline{G}}$ as a subgraph of $\widetilde{G}$. Since $\psi$ is a morphism from $\widetilde{G}$ to $\widetilde{\overline{G}}$, we have $\overline{\widetilde{G}} = \overline{\widetilde{\overline{G}}}$ by Lemma \ref{lem:retractfrommorph}.\end{proof}

\begin{lem}\label{lem:tildeandbaridempotent}
Let $G$ be an $X$-graph. Then $\widetilde{\widetilde{G}} = \widetilde{G}$ and $\overline{\overline{G}} = \overline{G}$.
\end{lem}

\begin{proof}
This is clear from the definitions.
\end{proof}

Note that unpruned multiplication and $(+)$ do not define a left adequate structure on the set of $X$-graphs. Indeed, the identity $G^{(+)}G = G$ does not hold for any non-trivial $G$ (our multiplication has no requirement for retraction to be performed). However, we may use these operations to define a left adequate structure of the set of idempath identified and retracted graphs resulting from trees.

We define $\textrm{PT}(C;X)$ to be the set of all (isomorphism types of) $X$-graphs of the form \[\mathrm{PT}(C;X) := \left.\left\{\overline{\widetilde{T}}\ \right|\ T \in \mathrm{FLAd}(X)\right\}.\] We call $\overline{\widetilde{T}}$ the \textit{pretzel} of $T$. We define a multiplication and unary operation $+$ on $\textrm{PT}(C;X)$ by \[GH := \overline{\widetilde{G \times H}} \textrm{ and } G^+ := \overline{\widetilde{G^{(+)}}}.\]These are well-defined operations on graphs in $\mathrm{PT}(C;X)$:

\begin{lem}\label{lem:comptildebar}
    For any trees $T,S \in \mathrm{FLAd}(X)$, $\overline{\widetilde{T}}\ \overline{\widetilde{S}} = \overline{\widetilde{T \times S}}$ and $\overline{\widetilde{T}}^+ = \overline{\widetilde{T^{(+)}}}$.
\end{lem}

\begin{proof}
    Firstly, \[\overline{\widetilde{T}}\ \overline{\widetilde{S}} = \overline{\widetilde{\overline{\widetilde{T}}\times\overline{\widetilde{S}}}} \stackrel{(1)}{=} \overline{\widetilde{\overline{\overline{\widetilde{T}}\times\overline{\widetilde{S}}}}} \stackrel{(2)}{=} \overline{\widetilde{\overline{\widetilde{T}\times\widetilde{S}}}} \stackrel{(3)}{=} \overline{\widetilde{\widetilde{T}\times\widetilde{S}}} \stackrel{(4)}{=} \overline{\widetilde{T \times S}}\]where $(2)$ and $(4)$ follow from Lemma \ref{lem:timeswithtildeandprune} and $(1)$ and $(3)$ follow from Lemma \ref{lem:stacking}.

    Secondly,\[\overline{\widetilde{T}}^+ = \overline{\widetilde{{\overline{\widetilde{T}}}^{(+)}}} \stackrel{(1)}{=} \overline{\widetilde{\overline{{\overline{\widetilde{T}}}^{(+)}}}} \stackrel{(2)}{=} \overline{\widetilde{\overline{{\widetilde{T}}^{(+)}}}} \stackrel{(3)}{=} \overline{\widetilde{\overline{\widetilde{T^{(+)}}}}} \stackrel{(4)}{=} \overline{\widetilde{\widetilde{T^{(+)}}}} \stackrel{(5)}{=} \overline{\widetilde{T^{(+)}}}\]where $(2)$ and $(3)$ follow from Lemma \ref{lem:pluswithtildeandprune}, $(1)$ and $(4)$ follow from Lemma \ref{lem:stacking} and $(5)$ follows from Lemma \ref{lem:tildeandbaridempotent}.\end{proof}

Recall the definition of $\textrm{FLAd}(X)$. Its operations are given by $TS := \overline{T \times S}$ and $T^+ := \overline{T^{(+)}}$ where we identify trees with their isomorphism types. Idempath identification and retraction are well-defined processes when extended to these isomorphism types of trees in $\textrm{FLAd}(X)$. In particular, we have the following result.

\begin{thm}\label{thm:PhifromUAtoPA}
    The map $\Phi : \mathrm{FLAd}(X) \to \mathrm{PT}(C;X)$ defined by $T \mapsto \overline{\widetilde{T}}$ is a surjective $(2,1,0)$-morphism.
\end{thm}

\begin{proof}
    By definition, any graph in $\mathrm{PT}(C;X)$ is $\Phi(T)$ for some tree $T$. Now let $T,S \in \mathrm{FLAd}(X)$. We show that $\Phi(TS) = \Phi(T)\Phi(S)$ and $\Phi\left(T^{+}\right) = (\Phi(T))^+$.
    
    Now by Lemma \ref{lem:comptildebar}, we have
    \begin{equation*}
        \Phi(T)\Phi(S) = \overline{\widetilde{T}}\ \overline{\widetilde{S}} = \overline{\widetilde{T \times S}} \stackrel{(\dagger)}{=} \overline{\widetilde{\overline{T \times S}}} = \Phi(TS)
    \end{equation*}
    and
    \begin{equation*}
        (\Phi(T))^+ = \overline{\widetilde{T}}^+ = \overline{\widetilde{T^{(+)}}} \stackrel{(\dagger)}{=} \overline{\widetilde{\overline{T^{(+)}}}} = \Phi\left(T^{(+)}\right).
    \end{equation*}where the equalities marked $(\dagger)$ follow from Lemma \ref{lem:stacking}.
    
    Finally, it is clear that the identity tree $1$ has no idempath identifications or retractions necessary, and thus $\Phi(1) = 1$. Therefore $\Phi$ is a surjective $(2,1,0)$-morphism.
\end{proof}

\begin{cor}\label{cor:PAmonoid}
The multiplication on $\mathrm{PT}(C;X)$ is associative, $+$ is an idempotent operation and $\mathrm{PT}(C;X)$ is $X$-generated (in the $(2,1,0)$ signature). The subsemigroup generated by the image of $+$ is a semilattice.
\end{cor}

\begin{proof}
    This follows from these properties of $\mathrm{FLAd}(X)$.
\end{proof}

We note that the definition of $\textrm{PT}(C;X)$ and the operations upon it depend only on which words in $C$ represent the identity. Hence by Corollary~\ref{cor:twosided} we have:
\begin{prop}\label{prop:twosided}
For every $X$-generated right cancellative monoid $C$, we have that $\mathrm{PT}(C;X)$ is isomorphic (as a $(2,1,0)$-algebra) to $\mathrm{PT}(C';X)$ for some $X$-generated two-sided cancellative monoid $C'$.
\end{prop}

Since the class of left adequate monoids does not form a variety, Theorem \ref{thm:PhifromUAtoPA} does not allow us to immediately conclude that $\mathrm{PT}(C;X)$ is left adequate. For this, a little more work is needed.

\begin{lem}\label{lem:x+xequalsx}
    $G^+G = G$ for any $G \in \mathrm{PT}(C;X)$.
\end{lem}

\begin{proof}
    Clearly $\overline{\widetilde{G}} = G$. By definition of the operations and Lemma \ref{lem:stacking}, we have\[G^+G = \overline{\widetilde{\overline{\widetilde{G^{(+)}}}\times \overline{\widetilde{G}}}} = \overline{\widetilde{\overline{\overline{\widetilde{G^{(+)}}}\times \overline{\widetilde{G}}}}} = \overline{\widetilde{\overline{\widetilde{G^{(+)}}\times\widetilde{G}}}} = \overline{\widetilde{\widetilde{G^{(+)}}\times\widetilde{G}}} = \overline{\widetilde{G^{(+)}\times G}} = \overline{\widetilde{\overline{G^{(+)}\times G}}}.\]Moreover, $\overline{G^{(+)}\times G} = \overline{G} = G$, by retracting the copy of $G^{(+)}$ onto the copy of $G$. Hence we have $G^+G = \overline{\widetilde{G}} = G$ as required.
\end{proof}

\begin{lem}\label{lem:abundant}
    Let $G,H,T \in \mathrm{PT}(C;X)$. Then $GT = HT \iff GT^+ = HT^+$.
\end{lem}

\begin{proof}
By Proposition~\ref{prop:twosided}, we may assume without loss of generality that $C$ is two-sided cancellative.
    Certainly if $G,H,T$ are such that $GT^+ = HT^+$, then by associativity and Lemma \ref{lem:x+xequalsx}, \[GT = G(T^+T) = (GT^+)T = (HT^+)T = H(T^+T) = HT.\]
    
        Conversely, suppose $GT = HT$. By definition of the multiplication, we have that $\overline{\widetilde{G \times T}} = GT = HT = \overline{\widetilde{H \times T}}$. Let $\rho_G$ and $\rho_H$ be the retraction maps for $\widetilde{G \times T}$ and $\widetilde{H \times T}$ respectively, both with images isomorphic to $\overline{\widetilde{G \times T}} = GT = HT = \overline{\widetilde{H \times T}}$.
    
    By definition of unpruned multiplication, $G \times T^{(+)}$ has the same underlying graph as $G \times T$ but with relocated end vertex. It therefore follows that $\widetilde{G \times T^{(+)}}$ has the same underlying graph as $\widetilde{G \times T}$, with endpoint the image of the amalgamated vertex of $G$ and $T$.
 
 We claim that that the map $\rho_G$ also defines a retract on $\widetilde{G \times T^{(+)}}$. Clearly, it is an idempotent morphism of labelled graphs and fixes
the start vertex. To establish the claim we need only to show that it fixes the end vertex, that is, the image of the amalgamated vertex of $G$ and $T$. Let $g$ be the end vertex, and suppose for a contradiction that $\rho_G(g) = h \neq k$. Let $g'$ be the amalgamated vertex and choose a pre-image $k'$ of $k$ in $G \times T^{(+)}$.
By Corollary~\ref{cor:tildekeepsCvalue}, we have $[g']_C = [k']_C$. Also, by Corollary~\ref{cor:tildekeepsCvalue}, if there was a directed path (in either direction) between $g'$ and $k'$ that path would have to be an idempath,
which would mean $g'$ and $k'$ were identified in $\widetilde{G \times T^{(+)}}$, contradicting the assumption that $g \neq k$. Hence, it will suffice to show that there
is such a directed path.

If $k'$ comes from the embedded copy of $T^{(+)}$ then then there is clearly a directed path from $g'$ to $k'$. So assume $k'$ comes from the embedded copy of $G$.
Let $e$ denote the endpoint of $T$, and choose a directed path from $g'$ to $e$ in $G \times T^{(+)}$, and let $p$ be the label of this path. Then there is also a
path labelled $p$ in $\widetilde{G \times T^{(+)}}$ from $g$ to the image of $e$. Since $\rho_G$ is a morphism, fixes $e$ and maps
$g$ to $k$, it follows that must also be a path in $\widetilde{G \times T^{(+)}}$ from $k$ to $e$ labelled $p$.

For this to exist, $G \times T^{(+)}$ must have an undirected path $w$ from $k'$ to $e$ of the form $x_1 y_2 x_2 y_2 \dots x_n y_n$ where each $x_i$ is a directed path traversed positively, and each $y_i$ is a constructable semiwalk (and $p$ is the label of $x_1 \cdots x_n$). This path must pass through $g'$, since $k'$ is in G, $e$ is in T, and $g'$ is the sole point of connection between these subgraphs. Now consider the image of this path in $\widetilde{G \times T^{(+)}}$. Clearly $x_1 x_2 \dots x_n$ becomes a directed path, and by Lemma \ref{lem:idempathsSCC}, every vertex traversed by the $y_i$’s lies in the same strongly connected component as a vertex of this directed path. It follows that every vertex on the path is at the end of a directed path starting at $k'$.
 
In particular, there is a directed path in $\widetilde{G \times T^{(+)}}$ from $k'$ to $g'$, as required.
This completes the proof of the claim that $\rho_G$ defines a retract on $\widetilde{G \times T^{(+)}}$.

A dual argument shows that $\rho_H$ defines a retract on $\widetilde{H \times T^{(+)}}$, and has image $HT = GT$ but with end vertex at the start of $T$. Therefore by confluence of retracts and Lemma \ref{lem:stacking}, we have that \[GT^+ = \overline{\widetilde{\overline{G \times T^{(+)}}}} = \overline{\widetilde{\overline{\rho_G\left(G \times T^{(+)}\right)}}} = \overline{\widetilde{\overline{\rho_H\left(H \times T^{(+)}\right)}}} = \overline{\widetilde{\overline{H \times T^{(+)}}}} = HT^+\]as required.
\end{proof}

\begin{lem}\label{lem:idemsareplus}
    Let $G \in \mathrm{PT}(C;X)$. Then $G$ is idempotent if and only if $G^+ = G$.
\end{lem}
    
\begin{proof}
    The right-to-left implication follows immediately from Corollary \ref{cor:PAmonoid}.

    Suppose $G \in \mathrm{PT}(C;X)$ has $GG = G$. Consider any path in $G$ from $\alpha(G)$ to $\omega(G)$, say labelled $w \in X^*$. Since $\sim$ and retraction are graph morphisms, in the graph $GG$ there exists a path from $\alpha(GG)$ to $\omega(GG)$ labelled $w^2$. Since $G$ is idempotent, $GG = G$ and thus there exists a path from $\alpha(G)$ to $\omega(G)$ labelled $w^2$. By Lemma \ref{lem:pathsintilde}, $[w]_C = [w^2]_C$. Hence $[w]_C = 1$ by right-cancellativity. Since $G$ is idempath identified, it follows that the path labelled $w$ in $G$ is a cycle, i.e. that $G$ has $\alpha(G) = \omega(G)$. In particular, $G = G^+$.
\end{proof}

\begin{cor}
    $\mathrm{PT}(C;X)$ is an $X$-generated left adequate monoid.
\end{cor}

\begin{proof}
    By Corollary \ref{cor:PAmonoid} and Lemma \ref{lem:idemsareplus}, is is enough to show that every element $T \in \mathrm{PT}(C;X)$ is $\mathcal{R}^*$-related to some idempotent. Indeed by Lemma \ref{lem:abundant}, $T \mathcal{R}^* T^+$ and by Lemma \ref{lem:idemsareplus}, $T^+$ is idempotent.
\end{proof}

We call $\mathrm{PT}(C;X)$ the \textit{pretzel monoid} of $C$ with respect to $X$. 

\begin{example}\label{ex:ptz2}
    Consider the cyclic group $C = C_2 = \mathrm{Mon}\langle x\ |\ x^2 = 1 \rangle$ of order $2$. Consider any tree $T \in \textrm{FLAd}(\{x\})$ and any branch of length at least $2$. Since all edges are labelled $x$, the branch will contain an idempath labelled by $x^2$ and thus will admit a non-trivial idempath identification. By considering all possibilities, one may see that the pretzel monoid $\mathrm{PT}(C_2;\{x\})$ consists of exactly $5$ graphs, namely the ones shown in Figure \ref{fig:ptz2}; that is any tree $T \in \textrm{FLAd}(\{x\})$ has its pretzel $\overline{\widetilde{T}}$ appearing in Figure \ref{fig:ptz2}.

    \begin{figure}[ht]
    \centering
        \begin{tikzpicture}
            \GraphInit[vstyle=Empty]
            \SetVertexSimple[MinSize = 1pt]
            \SetUpEdge[lw = 0.5pt]
            \tikzset{EdgeStyle/.style={->-}}
            \tikzset{VertexStyle/.append style = {minimum size = 3pt, inner sep = 0pt}}
            \SetVertexNoLabel
            \SetGraphUnit{2}

            \node (O) at ( 0,0) {};
            \node (A) at ( 0,1) {\large$\PlusCross$};

        \end{tikzpicture}
        \hspace{0.05\textwidth}
        \begin{tikzpicture}
            \GraphInit[vstyle=Empty]
            \SetVertexSimple[MinSize = 1pt]
            \SetUpEdge[lw = 0.5pt]
            \tikzset{EdgeStyle/.style={->-}}
            \tikzset{VertexStyle/.append style = {minimum size = 3pt, inner sep = 0pt}}
            \SetVertexNoLabel
            \SetGraphUnit{2}
            
            \node (A) at ( 0,0) {\large$+$};
            \node (B) at ( 1.5,2) {\large$\times$}; 
            
            \Edge(A)(B)\node (l) at ( 0.5,1.2) {$x$};
            
        \end{tikzpicture}
        \hspace{0.05\textwidth}
        \begin{tikzpicture}
            \GraphInit[vstyle=Empty]
            \SetVertexSimple[MinSize = 1pt]
            \SetUpEdge[lw = 0.5pt]
            \tikzset{EdgeStyle/.style={->-}}
            \tikzset{VertexStyle/.append style = {minimum size = 3pt, inner sep = 0pt}}
            \SetVertexNoLabel
            \SetGraphUnit{2}
            
            \node (A) at ( 0,0) {\large$\PlusCross$}; 
            \Vertex[x=1.5,y=2]{B}
            
            \Edge[style={bend left}](A)(B)\node (l) at ( 0.5,1.7) {$x$};
            \Edge[style={bend left}](B)(A)\node (l) at ( 1,0.3) {$x$};
            
        \end{tikzpicture}
        \hspace{0.05\textwidth}
        \begin{tikzpicture}
            \GraphInit[vstyle=Empty]
            \SetVertexSimple[MinSize = 1pt]
            \SetUpEdge[lw = 0.5pt]
            \tikzset{EdgeStyle/.style={->-}}
            \tikzset{VertexStyle/.append style = {minimum size = 3pt, inner sep = 0pt}}
            \SetVertexNoLabel
            \SetGraphUnit{2}
            
            \node (A) at ( 0,0) {\large$+$};
            \node (B) at ( 1.5,2) {\large$\times$}; 
            
            \Edge[style={bend left}](A)(B)\node (l) at ( 0.5,1.7) {$x$};
            \Edge[style={bend left}](B)(A)\node (l) at ( 1,0.3) {$x$};
            
        \end{tikzpicture}
        \hspace{0.05\textwidth}
        \begin{tikzpicture}
            \GraphInit[vstyle=Empty]
            \SetVertexSimple[MinSize = 1pt]
            \SetUpEdge[lw = 0.5pt]
            \tikzset{EdgeStyle/.style={->-}}
            \tikzset{VertexStyle/.append style = {minimum size = 3pt, inner sep = 0pt}}
            \SetVertexNoLabel
            \SetGraphUnit{2}
            
            \node (A) at ( 0,0) {\large$\PlusCross$};
            \Vertex[x=1.5,y=2]{B}
            
            \Edge(A)(B)\node (l) at ( 0.5,1.2) {$x$};
            
        \end{tikzpicture}
        \caption{The $5$ graphs in the pretzel monoid $\textrm{PT}(C_2;\left\{x\right\})$. From left-to-right, the pretzels are images of (e.g.) the trees $1,x,x^2,x^3$ and $x^+ \in \mathrm{FLAd}(x)$.}
        \label{fig:ptz2}
        \end{figure}
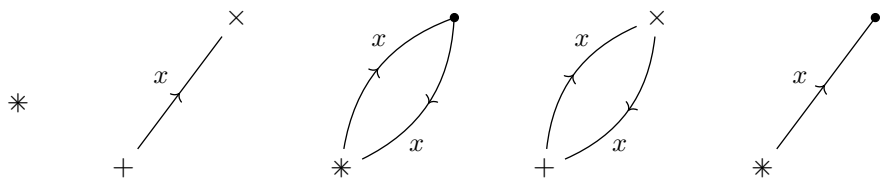

\end{example}

\section{A Presentation}\label{sec:presentation}

The realisation of pretzel monoids described in Section~\ref{sec:pretzel} provides a family of previously unknown geometrically described left adequate monoids. We encourage the reader to choose their favourite right cancellative monoid (or even group), and consider its pretzel monoid with a given generating set. Even small choices for $C$ produce interesting left adequate monoids. Recall Example \ref{ex:ptz2}: $\mathrm{PT}(C_2;x)$ is a $5$ elements left adequate monoid (in fact it is \textit{left ample} \cite{gould:rightcanc}) which is not inverse. One may also verify that $\textrm{PT}(C_3;\left\{x\right\})$ has $10$ elements, and is left adequate, but neither left ample nor inverse.

Moreover, one may compute the multiplication table and verify that $\textrm{PT}(C_2;\left\{x\right\})$ has a presentation given by \[\textrm{PT}(C_2;\left\{x\right\}) \cong \textrm{LAd}\left\langle x\ |\ (x^2)^+ = x^2\right\rangle.\]
We shall show, more generally, that for any $X$-generated right cancellative monoid $C$, the pretzel monoid $\textrm{PT}(C;X)$ is always isomorphic to:
\[\mathcal{M}(C;X) = \textrm{LAd}\left\langle X\ |\ w^+ = w \textrm{ for } w \in X^* \textrm{ such that } [w]_C = 1\right\rangle.\] By Corollary~\ref{cor:twosided}, 
we may replace $C$ with an $X$-generated two-sided cancellative monoid without changing set of words which represent the identity, and hence without changing
either $\textrm{PT}(C;X)$ and $\mathcal{M}(C;X)$.

Assume then, that $C$ is an $X$-generated two-sided cancellative monoid. For ease of notation, we write $\mathcal{M}$ for $\mathcal{M}(C;X)$. To show that $\textrm{PT}(C;X)$ and $\mathcal{M}$ represent the same monoid, our proof strategy is to show that both monoids $\textrm{PT}(C;X)$ and $\mathcal{M}$ are quotients of each other - one of these directions is easy.

\begin{thm}\label{thm:quotienteasy}
    The pretzel monoid $\textrm{PT}(C;X)$ is a quotient of $\mathcal{M}$.
\end{thm}

\begin{proof}
    Let $w \in X^*$ such that $[w]_C = 1$. As a tree in $\mathrm{FLAd}(X)$, $w$ is a single idempath, and $w^+$ is the same tree with identified start and end vertex. There are valid idempath identifications on both $w$ and on $w^+$ identifying the start and end of the path labelled $w$. Both of these identifications result in a graph consisting of a cycle labelled $w$ with identified start and end vertex and the start of $w$. Hence $w$ and $w^+$ have a common descendant, and thus $\widetilde{w} = \widetilde{w^+}$. Thus $[w]_{\mathrm{PT}(C;X)} = [w^+]_{\mathrm{PT}(C;X)}$.

    Since $\textrm{PT}(C;X)$ is also left adequate, it follows that $\textrm{PT}(C;X)$ satisfies all defining relations of $\mathcal{M}$. Thus $\textrm{PT}(C;X)$ is a quotient of $\mathcal{M}$.
\end{proof}

The remainder of this section is devoted to showing that $\mathcal{M}$ is a quotient of $\textrm{PT}(C;X)$. Our goal, therefore, is to show that if two $X$-trees $S$ and $T$ have $\overline{\widetilde{S}} = \overline{\widetilde{T}}$, then $[S]_\mathcal{M} = [T]_\mathcal{M}$.

Recall that Corollary \ref{cor:movethru} allows us to modify $X$-trees without changing their value in $\mathcal{M}$ by copying cones of vertices from one end of an idempath to the other. Notice that an idempath is a $1$-constructable semiwalk; one may hope to be able to copy cones through $n$-constructable semiwalks for any $n \geq 1$ without changing the value in $\mathcal{M}$. Indeed we are able to do so.

Given vertices $u,v$ of an $X$-tree $T$, define $T|^{u\to v}$ to be the $X$-tree obtained as follows. The underlying tree is obtained taking the disjoint $X$-tree $T$ and an extra copy of the subtree $\mathrm{Cone}_T(u)$, and identifying the vertex $u$ in the copied $\mathrm{Cone}_T(u)$ with the vertex $v$ in $T$. Informally, we `glue' a copy of $\mathrm{Cone}_T(u)$ to $v$. Define $\phi$ to be the copying map from $\mathrm{Cone}_T(u)$ to the obvious subtree of $\mathrm{Cone}_{T|^{u \to v}}(v)$. We identify $T$ with its natural image in the tree in $T|^{u\to v}$. We make $T|^{u\to v}$ an $X$-tree by defining $\alpha(T|^{u\to v})$ to be $\alpha(T)$, and

\[ \omega\left(T|^{u\to v}\right) = \begin{cases}
    \omega(T) & \textrm{if } \omega(T) \notin \mathrm{Cone}_T(u)\\
    \omega(\phi(\mathrm{Cone}_T(u))) & \textrm{if } \omega(T) \in \mathrm{Cone}_T(u)
    \end{cases}\]That is, the endpoint is copied along with the cone if possible.

\begin{example}
    Figure \ref{fig:copy} shows an $\{x,y\}$-tree $T$ with vertices $\mathbf{u},\mathbf{v}$ and $T|^{\mathbf{u}\to \mathbf{v}}$.
    \begin{figure}[ht]
    \centering
        {
        \begin{tikzpicture}
            \GraphInit[vstyle=Empty]
            \SetVertexSimple[MinSize = 1pt]
            \SetUpEdge[lw = 0.5pt]
            \tikzset{EdgeStyle/.style={->-}}
            \tikzset{VertexStyle/.append style = {minimum size = 3pt, inner sep = 0pt}}
            \SetVertexNoLabel
            \SetGraphUnit{2}
            
            \node (A) at ( 0,0) {\large$+$};
            \Vertex[x=0,y=1]{U};
            \Vertex[x=1,y=2]{R1}\node (l) at ( 1.3,2) {$\mathbf{u}$};
            \node (R2) at ( 1,3) {\large$\times$};

            \Vertex[x=-1,y=2]{L1}\node (l) at ( -1.2,1.8) {$\mathbf{v}$};
            \Vertex[x=-2,y=3]{L2};
    
            \Vertex[x=0,y=3]{LR1};

            \Edge(A)(U)\draw (A) -- (U) node [midway, left=2pt] {$x$};
            \Edge(U)(R1)\draw (U) -- (R1) node [midway, right=2pt,below=2pt] {$x$};
            \Edge(R1)(R2)\draw (R1) -- (R2) node [midway, left=2pt] {$y$};
            \Edge(U)(L1)\draw (U) -- (L1) node [midway, left=2pt,below=2pt] {$x$};
            \Edge(L1)(LR1)\draw (L1) -- (LR1) node [midway, right=2pt,below=2pt] {$y$};
            \Edge(L1)(L2)\draw (L1) -- (L2) node [midway, left=2pt,below=2pt] {$x$};
    
        \end{tikzpicture}
        }\hspace{0.05\textwidth}
        {
        \begin{tikzpicture}
            \GraphInit[vstyle=Empty]
            \SetVertexSimple[MinSize = 1pt]
            \SetUpEdge[lw = 0.5pt]
            \tikzset{EdgeStyle/.style={->-}}
            \tikzset{VertexStyle/.append style = {minimum size = 3pt, inner sep = 0pt}}
            \SetVertexNoLabel
            \SetGraphUnit{2}
            
            \node (A) at ( 0,0) {\large$+$};
            \Vertex[x=0,y=1]{U};
            \Vertex[x=1,y=2]{R1};
            \Vertex[x=1,y=3]{R2};
            \node (new) at ( -1,3.5) {\large$\times$}; 
            \Vertex[x=-1,y=2]{L1};
            \Vertex[x=-2,y=3]{L2};    
            \Vertex[x=0,y=3]{LR1};
    
            \Edge(A)(U)\draw (A) -- (U) node [midway, left=2pt] {$x$};
            \Edge(U)(R1)\draw (U) -- (R1) node [midway, right=2pt,below=2pt] {$x$};
            \Edge(R1)(R2)\draw (R1) -- (R2) node [midway, left=2pt] {$y$};
            \Edge(U)(L1)\draw (U) -- (L1) node [midway, left=2pt,below=2pt] {$x$};
            \Edge(L1)(LR1)\draw (L1) -- (LR1) node [midway, right=2pt,below=2pt] {$y$};
            \Edge(L1)(L2)\draw (L1) -- (L2) node [midway, left=2pt,below=2pt] {$x$};
            \Edge(L1)(new)\draw (L1) -- (new) node [midway, right=5pt,above=1pt] {$y$};
    
        \end{tikzpicture}
        }
    \caption{An $\{x,y\}$-tree $T$ and the tree $T|^{\mathbf{u}\to \mathbf{v}}$.}
    \label{fig:copy}
    \end{figure}
\end{example}

\begin{thm}\label{thm:copy}
    Let $C$ be an $X$-generated special right cancellative monoid. Let $T$ be an $X$-tree and suppose $u$ and $v$ are vertices of $T$ which are identified in $\widetilde{T}$. Then $[T]_\mathcal{M} = [T|^{u\to v}]_\mathcal{M}$.
\end{thm}
\begin{proof}
If $u$ and $v$ are the same vertex in $T$, then $[T]_{\mathrm{FLAd}(X)} = [T|^{u\to v}]_{\mathrm{FLAd}(X)}$ so certainly the result holds. Now suppose $u$ and $v$ are distinct vertices of $T$. By Theorem \ref{thm:whenidentify}, there exists an $n$-constructable semiwalk $\mathfrak{s}$ on $T$ from $u$ to $v$ for some $n \geq 1$. If there exists a directed path from $u$ to $v$ or from $v$ to $u$, then this path is an idempath by Lemma \ref{lem:pathsintilde}. Hence the result follows directly from Corollary \ref{cor:movethru}.

We show our result by induction on $n$. If $n=1$, then $\mathfrak{s}$ is a directed path between $u$ and $v$ and the result follows as above. Assume now that the result holds for $K$-constructable semiwalks for some integer $K \geq 1$. Suppose that $\mathfrak{s}$ is $(K+1)$-constructable, and moreover that there is no directed path from $u$ to $v$ or from $v$ to $u$ in $T$.

By Corollary \ref{cor:twosided} and Lemma \ref{lem:standard}, we may assume that $\mathfrak{s}$ is a standard semiwalk. Let $t$ denote the lowest vertex of the tree traversed by $\mathfrak{s}$. Note that $t \notin \{u,v\}$, as there exists paths from $t$ to $u$ and $t$ to $v$. Decompose $\mathfrak{s} = p\mathfrak{w}q$ into semiwalks $p,\mathfrak{w},q$, where $\lambda(\mathfrak{w}) \in X^* \cup (X^{-1})^*$ is a final inserted word in a $(K+1)$-construction of $\lambda(\mathfrak{s})$ which corresponds to an idempath $w$ in $T$ with $\alpha(w)=t$. Note that $\mathfrak{w}$ may traverse $w$ either positively or negatively. Moreover, note that $\lambda(p)\lambda(q)$ is a $K$-constructable word.

Denote $\mathrm{Cone}_T(t)$ by $A$. Form the tree $T|^{t \to t}$, and denote by $A'$ the newly added copy of $A$. Since there exists a morphism of $T|^{t \to t}$ to $T$ and $T$ is realisable as a birooted subtree, we have $[T]_{\mathrm{FLAd}(X)} = [T|^{t \to t}]_{\mathrm{FLAd}(X)}$ by Lemma \ref{lem:retractfrommorph}.

Define the map $\phi:A \to A'$ to be the obvious isomorphism of directed graphs. Since $A$ contains the vertices $u$ and $v$, the path $w$ and the semiwalks $p$ and $q$, $A'$ contains vertices $\phi(u)$ and $\phi(v)$, the path $\phi(w)$, and semiwalks $\phi(p)$ and $\phi(q)$. Moreover, $\mathrm{Cone}_{T|^{t \to t}}(\phi(u)) \cong \mathrm{Cone}_T(u)$. Figure \ref{fig:TandTprime} shows the general case of $T$ and $T|^{t \to t}$, with subtrees $A$ and $A'$ denoted. In both cases, the endpoint of $T|^{t \to t}$ may or may not be located in $A'$, but it is certainly not located in $A$.

\begin{figure}[ht]
    \centering
    \resizebox{0.8\textwidth}{!}{%
    \begin{circuitikz}
        \tikzstyle{every node}=[font=\Huge]
        \draw [ fill={rgb,255:red,0; green,0; blue,0} ] (10,14.25) circle (0.25cm);
        \draw [ fill={rgb,255:red,0; green,0; blue,0} ] (10,18) circle (0.25cm);
        \draw [->, >=Stealth, line width=1pt] (10,14.5) -- (10,17.75)node[pos=0.5, fill=white]{$w$};
        \draw [ fill={rgb,255:red,0; green,0; blue,0} ] (10,21.75) circle (0.25cm);
        \draw [ fill={rgb,255:red,0; green,0; blue,0} ] (6.25,18) circle (0.25cm);
        \draw [->, >=Stealth, dashed, line width=1pt] (6.25,18) -- (9.75,14.5)node[pos=0.5, fill=white]{$p$};
        \draw [->, >=Stealth, dashed, line width=1pt] (10,18) -- (10,21.5)node[pos=0.5, fill=white]{$q$};
        \node [font=\Huge] at (10,13.5) {$t$};
        \node [font=\Huge] at (6.25,17.25) {$u$};
        \node [font=\Huge] at (9.25,21.75) {$v$};
        \draw [ fill={rgb,255:red,0; green,0; blue,0} ] (18.25,14.25) circle (0.25cm);
        \draw [ fill={rgb,255:red,0; green,0; blue,0} ] (15.75,18) circle (0.25cm);
        \draw [->, >=Stealth, line width=1pt] (18.25,14.25) -- (16,17.75)node[pos=0.5, fill=white]{$w$};
        \draw [ fill={rgb,255:red,0; green,0; blue,0} ] (13.25,21.75) circle (0.25cm);
        \draw [ fill={rgb,255:red,0; green,0; blue,0} ] (13.25,16.75) circle (0.25cm);
        \draw [->, >=Stealth, dashed, line width=1pt] (13.25,16.75) -- (18,14.25)node[pos=0.5, fill=white]{$p$};
        \draw [->, >=Stealth, dashed, line width=1pt] (15.75,18) -- (13.5,21.5)node[pos=0.5, fill=white]{$q$};
        \node [font=\Huge] at (18.25,13.5) {$t$};
        \node [font=\Huge] at (13.25,17.5) {$u$};
        \node [font=\Huge] at (14.25,21.75) {$v$};
        \draw [ fill={rgb,255:red,0; green,0; blue,0} ] (20.75,18) circle (0.25cm);
        \draw [ fill={rgb,255:red,0; green,0; blue,0} ] (23.25,16.75) circle (0.25cm);
        \draw [ fill={rgb,255:red,0; green,0; blue,0} ] (23.25,21.75) circle (0.25cm);
        \draw [->, >=Stealth, line width=1pt] (18.25,14.25) -- (20.5,17.75)node[pos=0.5, fill=white]{$\phi(w)$};
        \draw [->, >=Stealth, dashed, line width=1pt] (20.75,18) -- (23,21.5)node[pos=0.5, fill=white]{$\phi(q)$};
        \draw [->, >=Stealth, dashed, line width=1pt] (23.25,16.75) -- (18.5,14.25)node[pos=0.5, fill=white]{$\phi(p)$};
        \node [font=\Huge] at (23.25,17.5) {$\phi(u)$};
        \node [font=\Huge] at (22.25,21.75) {$\phi(v)$};
        \draw [ color={rgb,255:red,0; green,0; blue,255}, line width=2pt, short] (18.25,14.25) -- (19.5,23);
        \draw [ color={rgb,255:red,0; green,0; blue,255}, line width=2pt, short] (18.25,14.25) -- (24.5,14.25);
        \draw [ color={rgb,255:red,255; green,0; blue,0}, line width=2pt, short] (18.25,14.25) -- (17,23);
        \draw [ color={rgb,255:red,255; green,0; blue,0}, line width=2pt, short] (18.25,14.25) -- (12,14.25);
        \node [font=\Huge, color={rgb,255:red,255; green,0; blue,0}] at (12.75,14.75) {$A$};
        \node [font=\Huge, color={rgb,255:red,0; green,0; blue,255}] at (24,14.75) {$A'$};
    \end{circuitikz}
    }\\\resizebox{0.8\textwidth}{!}{%
    \begin{circuitikz}
        \tikzstyle{every node}=[font=\Huge]
        \draw [ fill={rgb,255:red,0; green,0; blue,0} ] (31.25,14.25) circle (0.25cm);
        \draw [ fill={rgb,255:red,0; green,0; blue,0} ] (31.25,18) circle (0.25cm);
        \draw [->, >=Stealth, line width=1pt] (31.25,14.5) -- (31.25,17.5)node[pos=0.5, fill=white]{$w$};
        \draw [ fill={rgb,255:red,0; green,0; blue,0} ] (31.25,21.75) circle (0.25cm);
        \draw [ fill={rgb,255:red,0; green,0; blue,0} ] (27.5,18) circle (0.25cm);
        \draw [->, >=Stealth, dashed, line width=1pt] (31.25,21.75) -- (31.25,18.5)node[pos=0.5, fill=white]{$p$};
        \draw [->, >=Stealth, dashed, line width=1pt] (31.25,14.25) -- (27.75,17.75)node[pos=0.5, fill=white]{$q$};
        \node [font=\Huge] at (31.25,13.5) {$t$};
        \node [font=\Huge] at (30.5,21.75) {$u$};
        \node [font=\Huge] at (27.5,17.25) {$v$};
        \draw [ fill={rgb,255:red,0; green,0; blue,0} ] (39.5,14.25) circle (0.25cm);
        \draw [ fill={rgb,255:red,0; green,0; blue,0} ] (37,18) circle (0.25cm);
        \draw [->, >=Stealth, line width=1pt] (39.5,14.25) -- (37.25,17.75)node[pos=0.5, fill=white]{$w$};
        \draw [ fill={rgb,255:red,0; green,0; blue,0} ] (34.75,21.75) circle (0.25cm);
        \draw [ fill={rgb,255:red,0; green,0; blue,0} ] (34.75,16.75) circle (0.25cm);
        \draw [->, >=Stealth, dashed, line width=1pt] (34.75,21.75) -- (36.75,18.5)node[pos=0.5, fill=white]{$p$};
        \draw [->, >=Stealth, dashed, line width=1pt] (39.5,14.25) -- (35,16.5)node[pos=0.5, fill=white]{$q$};
        \node [font=\Huge] at (39.5,13.5) {$t$};
        \node [font=\Huge] at (35.5,21.75) {$u$};
        \node [font=\Huge] at (34.5,17.5) {$v$};
        \draw [ fill={rgb,255:red,0; green,0; blue,0} ] (42,18) circle (0.25cm);
        \draw [ fill={rgb,255:red,0; green,0; blue,0} ] (44.5,16.75) circle (0.25cm);
        \draw [ fill={rgb,255:red,0; green,0; blue,0} ] (44.5,21.75) circle (0.25cm);
        \draw [->, >=Stealth, line width=1pt] (39.5,14.25) -- (41.75,17.75)node[pos=0.5, fill=white]{$\phi(w)$};
        \draw [->, >=Stealth, dashed, line width=1pt] (39.5,14.25) -- (44.25,16.5)node[pos=0.5, fill=white]{$\phi(q)$};
        \draw [->, >=Stealth, dashed, line width=1pt] (44.5,21.75) -- (42.25,18.5)node[pos=0.5, fill=white]{$\phi(p)$};
        \node [font=\Huge] at (43.5,21.75) {$\phi(u)$};
        \node [font=\Huge] at (44.5,17.5) {$\phi(v)$};
        \draw [ color={rgb,255:red,0; green,0; blue,255}, line width=2pt, short] (39.5,14.25) -- (40.75,23);
        \draw [ color={rgb,255:red,0; green,0; blue,255}, line width=2pt, short] (39.5,14.25) -- (45.75,14.25);
        \draw [ color={rgb,255:red,255; green,0; blue,0}, line width=2pt, short] (39.5,14.25) -- (38.25,23);
        \draw [ color={rgb,255:red,255; green,0; blue,0}, line width=2pt, short] (39.5,14.25) -- (33.25,14.25);
        \node [font=\Huge, color={rgb,255:red,255; green,0; blue,0}] at (34,14.75) {$A$};
        \node [font=\Huge, color={rgb,255:red,0; green,0; blue,255}] at (45.25,14.75) {$A'$};
    \end{circuitikz}
    }%
    \caption{The general situation of $T$ (left) and $T|^{t \to t}$ (right) when $w$ is positively traversed (top) and negatively traversed (bottom).}
    \label{fig:TandTprime}
\end{figure}
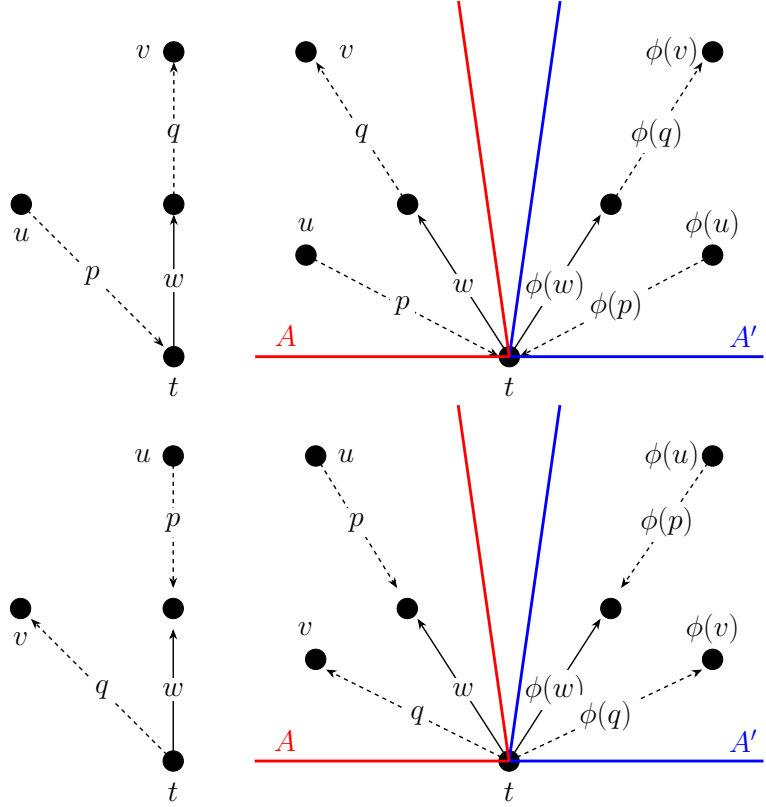

Define the tree $T_{A}$ to be the tree $T|^{t \to t}$ with the entire subtree $A$ moved up to $\omega(\phi(w))$. Similarly, define the tree $T_{A'}$ to be the tree $T|^{t \to t}$ with the entire subtree $A'$ moved up to $\omega(w)$. In both cases, if $\alpha(T)=t$ then we leave the start point of $T|^{t \to t}$ at $t$ when we move, though we move the endpoint with the subtree if applicable. By Corollary \ref{cor:movethru}, we have both $[T_{A}]_\mathcal{M} = [T|^{t \to t}]_\mathcal{M} = 
[T_{A'}]_\mathcal{M}$.

Note that $T_A$ and $T_{A'}$ have isomorphic underlying trees, only differing in the location of endpoint, specifically when $\omega(T) \in \mathrm{Cone}_T(t)$. The general idea of the proof is as follows: in the trees $T_A$ and $T_{A'}$, we may read a $K$-constructable semiwalk labelled $\lambda(p)\lambda(q)$ from $u$ or $\phi(u)$ to $v$ or $\phi(v)$, and appeal to the inductive hypothesis to copy certain cones. We then reappeal to Corollary \ref{cor:movethru} to move our subtrees $A$ or $A'$ back to $t$, and then show they may retract in our new tree. Our upcoming splitting of cases is solely to handle the location of the endpoint in our final tree.

Formally, we must now split into cases depending on $w$ and the location of $\omega(T)$.
    
\textbf{\textit{Case 1: $w$ is traversed positively by $\mathfrak{w}$ and $\omega(T) \in \mathrm{Cone}_T(u)$.}}\\Note that $\omega(T_{A'}) \in \mathrm{Cone}_{T_{A'}}(\phi(u))$. Since $w$ was traversed positively, in $T_{A'}$ there is a semiwalk \[\phi(u) \xrightarrow{\lambda(\phi(p))} \omega(w) \xrightarrow{\lambda(q)} v\] with label $\lambda(\phi(p))\lambda(q) = \lambda(p)\lambda(q)$. Since $\lambda(p)\lambda(q)$ is $K$-constructable, our inductive hypothesis applies and thus we have $[T_{A'}]_\mathcal{M} = [T_{A'}|^{\phi(u) \to v}]_\mathcal{M}$. In the tree $T_{A'}|^{\phi(u) \to v}$, the subgraph $A'$ now certainly does not contain the endpoint, as it was copied over with $\mathrm{Cone}_{T_{A'}}(\phi(u))$. We may now `undo' our move of $A'$ and move it back to $t$; by Corollary \ref{cor:movethru}, our $\mathcal{M}$-value is still unchanged. It follows that now $A'$ will retract onto $A$ and our new copy of $\mathrm{Cone}_{T_{A'}}(\phi(u)) \cong \mathrm{Cone}_T(u)$. This results in exactly the tree $T|^{u\to v}$.

\textbf{\textit{Case 2: $w$ is traversed positively by $\mathfrak{w}$ and $\omega(T) \notin \mathrm{Cone}_T(u)$.}}\\Since $w$ was traversed positively, in $T_{A}$ there is a semiwalk \[u \xrightarrow{\lambda(p)} \omega(\phi(w)) \xrightarrow{\lambda(\phi(q))} \phi(v)\] with label $\lambda(p)\lambda(\phi(q)) = \lambda(p)\lambda(q)$. Since $\lambda(p)\lambda(q)$ is $K$-constructable, our inductive hypothesis applies and thus we have $[T_{A}]_\mathcal{M} = [T_{A}|^{u \to \phi(v)}]_\mathcal{M}$. We may now `undo' our move of $A$ and move it back to $t$; by Corollary \ref{cor:movethru}, our $\mathcal{M}$-value is still unchanged. It follows that $A$ will retract onto $A'$ as it could have in $T|^{t \to t}$. This results in exactly the tree $T|^{u\to v}$.

\textbf{\textit{Case 3: $w$ is traversed negatively by $\mathfrak{w}$ and $\omega(T) \in \mathrm{Cone}_T(u)$.}}\\Note that $\omega(T_{A}) \in \mathrm{Cone}_{T_{A}}(\phi(u))$. Since $w$ was traversed negatively, in $T_{A}$ there is a semiwalk \[\phi(u) \xrightarrow{\lambda(\phi(p))} \alpha(w) \xrightarrow{\lambda(q)} v\] with label $\lambda(\phi(p))\lambda(q) = \lambda(p)\lambda(q)$. Since $\lambda(p)\lambda(q)$ is $K$-constructable, our inductive hypothesis applies and thus we have $[T_{A}]_\mathcal{M} = [T_{A}|^{\phi(u) \to v}]_\mathcal{M}$. In the tree $T_{A}|^{\phi(u) \to v}$, the subgraph $A'$ now certainly does not contain the endpoint, as it was copied over with $\mathrm{Cone}_{T_{A}}(\phi(u))$. We may now `undo' our move of $A$ and move it back to $t$; by Corollary \ref{cor:movethru}, our $\mathcal{M}$-value is still unchanged. It follows that now $A'$ will retract onto $A$ and our new copy of $\mathrm{Cone}_{T_{A}}(\phi(u)) \cong \mathrm{Cone}_T(u)$. This results in exactly the tree $T|^{u\to v}$.

\textbf{\textit{Case 4: $w$ is traversed negatively by $\mathfrak{w}$ and $\omega(T) \notin \mathrm{Cone}_T(u)$.}}\\Since $w$ was traversed negatively, in $T_{A'}$ there is a semiwalk \[u \xrightarrow{\lambda(p)} \omega(w) \xrightarrow{\lambda(\phi(q))} \phi(v)\] with label $\lambda(p)\lambda(\phi(q)) = \lambda(p)\lambda(q)$. Since $\lambda(p)\lambda(q)$ is $K$-constructable, our inductive hypothesis applies and thus we have $[T_{A'}]_\mathcal{M} = [T_{A'}|^{u \to \phi(v)}]_\mathcal{M}$. We may now `undo' our move of $A'$ and move it back to $t$; by Corollary \ref{cor:movethru}, our $\mathcal{M}$-value is still unchanged. It follows that $A$ will retract onto $A'$ as it could have in $T|^{t \to t}$. This results in exactly the tree $T|^{u\to v}$.

In all cases, we have constructed our desired tree. The result therefore follows by induction.\end{proof}

Theorem \ref{thm:copy} is our first clue that the pretzel monoid truly is $\mathcal{M}$. Notice that if $u$ and $v$ are identified vertices in $\widetilde{T}$, then the graphs $T$ and $T|^{u \to v}$ have isomorphic retracts after the vertices $u$ and $v$ are identified; in particular $\overline{\widetilde{T}} = \overline{\widetilde{T|^{u \to v}}}$.

Our strategy for showing $\overline{\widetilde{S}} = \overline{\widetilde{T}}$ implies $[S]_\mathcal{M}= [T]_\mathcal{M}$ involves building a tree $\mathcal{U}$ depending only on the pretzel $\overline{\widetilde{S}} = \overline{\widetilde{T}}$, and showing that $[S]_\mathcal{M} = [\mathcal{U}]_\mathcal{M} = [T]_\mathcal{M}$. This $\mathcal{U}$ will be the tree of \textit{almost simple paths} of $\overline{\widetilde{S}}$ ($=\overline{\widetilde{T}}$).

Fix a tree $T$ and its pretzel $\overline{\widetilde{T}}$. Suppose $\overline{\widetilde{T}}$ has $n$ edges, and denote them by the set $\mathcal{E}:= \{1,2,\dots,n\}$. An \textit{almost simple path} in $\overline{\widetilde{T}}$ is a directed path starting at the root which does not repeat vertices, except possibly its final visited vertex.

In general, these are paths with label such that any subword equal to $1$ in $C$ is a suffix. Each of these paths admit a labelling by $X$ and also by the set $\mathcal{E}$ where the label of an edge is determined by the index of the edge it was read from in $\overline{\widetilde{T}}$.

Here, it is necessary to introduce \textit{determinisation of trees} (see, for example, \cite{cfnb:geometry}). An $X$-tree $T$ is $X$-\textit{deterministic} if for every $X \in X$, any subgraph of $T$ with equally-labelled edges $e_1$ and $e_2$ of the form
\begin{equation}\label{eq:detmin}
    \begin{tikzpicture}
        \GraphInit[vstyle=Empty]
        \SetVertexSimple[MinSize = 1pt]
        \SetUpEdge[lw = 0.5pt]
        \tikzset{EdgeStyle/.style={->-}}
        \tikzset{VertexStyle/.append style = {minimum size = 3pt, inner sep = 0pt}}
        \SetVertexNoLabel
        \SetGraphUnit{2}
        
        \Vertex[x=0,y=1]{l};
        \Vertex[x=1,y=1.5]{t};
        \Vertex[x=1,y=0.5]{b};

        \Edge(l)(t)\draw (l) -- (t) node [midway, above=2pt] {$x$};
        \Edge(l)(b)\draw (l) -- (b) node [midway, below=2pt] {$x$}; 
    \end{tikzpicture}
\end{equation} implies that the $e_1$ and $e_2$ are the same edge of $T$. The \textit{determinisation} of $T$ with respect to the labelling $X$, is the deterministic tree obtained by sequentially transforming all subgraphs of the form \eqref{eq:detmin} to a single edge labelled $x$, as such the cone of the identified vertex becomes the union of two cones. This process is referred to as performing \textit{Stallings foldings} \cite{kapovich:fold, stallings:topology}. Such a graph is unique (see, for example, \cite[Theorem 4.4]{stephen:applications}).

The tree $\mathcal{U}$ is taken to be the (unrooted) tree consisting of all of the almost simple paths of $\overline{\widetilde{T}}$, determinised with respect to the labelling by $\{1,\dots,n\}$. Note that $\mathcal{U}$ may not be deterministic with respect to $X$. 

Each vertex or edge of $\mathcal{U}$ corresponds to a unique vertex or edge of $\overline{\widetilde{T}}$ with the expected correspondence. Since $\omega(T)$ may lie on multiple almost simple paths, there may be multiple vertices of $\mathcal{U}$ which correspond to $\omega(T)$. Because of this, we formally define the set \[\textbf{U}(\overline{\widetilde{T}}) := \{(\mathcal{U},v)\ |\ v \in \mathrm{V}(\mathcal{U}) \textrm{ corresponds to }\omega(T)\}.\]We represent $\mathbf{U}(\overline{\widetilde{T}})$ as a single tree, with multiple endpoints corresponding to each $(\mathcal{U},v) \in \mathbf{U}(\overline{\widetilde{T}})$ where we treat $v$ as the endpoint of $\mathcal{U}$. This tree admits both $X$- and $\mathcal{E}$-labellings via the respective labellings for $\mathcal{U}$. When referring to the labelling of edges, we will modify our $\lambda$ notation to $\lambda_X$ or $\lambda_{\mathcal{E}}$ to avoid ambiguity.

\begin{example}\label{ex:nonuniqU}
    Let $C = \mathbb{Z}_3 \times \mathbb{Z}_3 = \langle x,y \rangle$. Figure \ref{fig:pretzellabel} shows a pretzel $\Gamma$ with an $\mathcal{E}$-labelling where $\mathcal{E} = \left\{1,\dots,10\right\}$. We have the three trees of almost simple paths, shown in Figure \ref{fig:u}. Note that $\textbf{U}(\Gamma)$ is deterministic with this $\mathcal{E}$-labelling.
    \begin{figure}[ht]
    \centering
        {
        \begin{tikzpicture}
            \GraphInit[vstyle=Empty]
            \SetVertexSimple[MinSize = 1pt]
            \SetUpEdge[lw = 0.5pt]
            \tikzset{EdgeStyle/.style={->-}}
            \tikzset{VertexStyle/.append style = {minimum size = 3pt, inner sep = 0pt}}
            \SetVertexNoLabel
            \SetGraphUnit{2}
            
            \node (A) at ( 0,0) {\large$+$};
            \Vertex[x=0,y=1]{U1}
            \Vertex[x=0,y=2]{U2}
            \node (U3) at ( 0,3) {\large$\times$};
            \Vertex[x=0,y=4]{U4}
            \Vertex[x=0,y=5]{U5}
            \Vertex[x=-2,y=4]{UL}
            \Vertex[x=-2,y=2]{DL}
    
            \Edge(A)(U1)\draw (A) -- (U1) node [midway, right=2pt] {$x$};
            \Edge(U1)(U2)\draw (U1) -- (U2) node [midway, right=2pt] {$x$};
            \Edge(U2)(U3)\draw (U2) -- (U3) node [midway, right=2pt] {$y$};
            \Edge(U3)(U4)\draw (U3) -- (U4) node [midway, right=2pt] {$x$};
            \Edge(U4)(U5)\draw (U4) -- (U5) node [midway, right=2pt] {$y$};
            \Edge(DL)(U1)\draw (DL) -- (U1) node [midway, below=2pt] {$y$};
            \Edge(UL)(DL)\draw (UL) -- (DL) node [midway, left=2pt] {$y$};
            \Edge(U1)(UL)\draw (U1) -- (UL) node [midway, left=2pt] {$y$};
            \Edge(U4)(UL)\draw (U4) -- (UL) node [midway, above=2pt] {$x$};
            \Edge(UL)(U3)\draw (UL) -- (U3) node [midway, above=2pt,right=2pt] {$x$};
      
        \end{tikzpicture}
         }\hspace{0.05\textwidth}
        {
        \begin{tikzpicture}
            \GraphInit[vstyle=Empty]
            \SetVertexSimple[MinSize = 1pt]
            \SetUpEdge[lw = 0.5pt]
            \tikzset{EdgeStyle/.style={->-}}
            \tikzset{VertexStyle/.append style = {minimum size = 3pt, inner sep = 0pt}}
            \SetVertexNoLabel
            \SetGraphUnit{2}
            
            \node (A) at ( 0,0) {\large$+$};
            \Vertex[x=0,y=1]{U1}
            \Vertex[x=0,y=2]{U2}
            \node (U3) at ( 0,3) {\large$\times$};
            \Vertex[x=0,y=4]{U4}
            \Vertex[x=0,y=5]{U5}
            \Vertex[x=-2,y=4]{UL}
            \Vertex[x=-2,y=2]{DL}
    
            \Edge(A)(U1)\draw (A) -- (U1) node [midway, right=2pt] {$1$};
            \Edge(U1)(U2)\draw (U1) -- (U2) node [midway, right=2pt] {$2$};
            \Edge(U2)(U3)\draw (U2) -- (U3) node [midway, right=2pt] {$3$};
            \Edge(U3)(U4)\draw (U3) -- (U4) node [midway, right=2pt] {$4$};
            \Edge(U4)(U5)\draw (U4) -- (U5) node [midway, right=2pt] {$5$};
            \Edge(DL)(U1)\draw (DL) -- (U1) node [midway, below=2pt] {$8$};
            \Edge(UL)(DL)\draw (UL) -- (DL) node [midway, left=2pt] {$7$};
            \Edge(U1)(UL)\draw (U1) -- (UL) node [midway, left=2pt] {$6$};
            \Edge(U4)(UL)\draw (U4) -- (UL) node [midway, above=2pt] {$10$};
            \Edge(UL)(U3)\draw (UL) -- (U3) node [midway, above=2pt,right=2pt] {$9$};
      
        \end{tikzpicture}
        }
    \caption{A relabelling of a pretzel.}
    \label{fig:pretzellabel}
    \end{figure}

    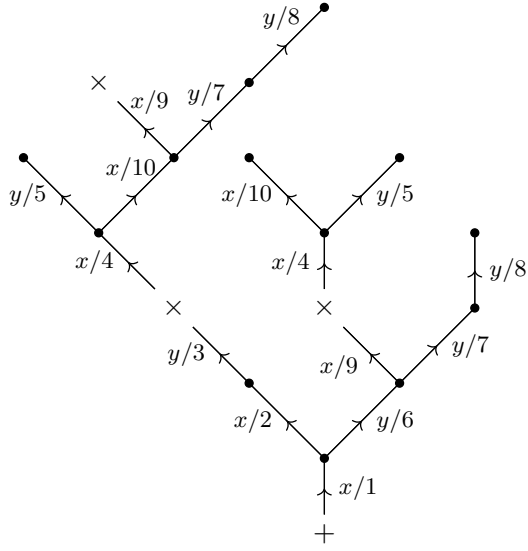
\begin{figure}[ht]
    \centering
    \begin{tikzpicture}
        \GraphInit[vstyle=Empty]
        \SetVertexSimple[MinSize = 1pt]
        \SetUpEdge[lw = 0.5pt]
        \tikzset{EdgeStyle/.style={->-}}
        \tikzset{VertexStyle/.append style = {minimum size = 3pt, inner sep = 0pt}}
        \SetVertexNoLabel
        \SetGraphUnit{2}
        
        \node (A) at ( 0,0) {\large$+$};
        \Vertex[x=0,y=1]{U}
        \Vertex[x=-1,y=2]{L}
        \node (LL) at ( -2,3) {\large$\times$};
        \Vertex[x=-3,y=4]{LLL}
        \Vertex[x=-4,y=5]{LLLL}
        \Vertex[x=-2,y=5]{LLLR}
        \node (LLLRL) at ( -3,6) {\large$\times$};
        \Vertex[x=-1,y=6]{LLLRR}
        \Vertex[x=0,y=7]{LLLRRR}

        \Vertex[x=1,y=2]{R}
        \Vertex[x=2,y=3]{RR}
        \Vertex[x=2,y=4]{RRR}
        \node (RL) at ( 0,3) {\large$\times$};
        \Vertex[x=0,y=4]{RLU}
        \Vertex[x=-1,y=5]{RLUL}
        \Vertex[x=1,y=5]{RLUR}

        \Edge(A)(U)\draw (A) -- (U) node [midway, right=2pt] {\small{$x / 1$}};
        
        \Edge(U)(L)\draw (U) -- (L) node [midway, left=2pt] {\small{$x / 2$}};
        \Edge(L)(LL)\draw (L) -- (LL) node [midway, left=2pt] {\small{$y / 3$}};
        \Edge(LL)(LLL)\draw (LL) -- (LLL) node [midway, left=2pt] {\small{$x / 4$}};
        \Edge(LLL)(LLLL)\draw (LLL) -- (LLLL) node [midway, left=2pt] {\small{$y / 5$}};
        \Edge(LLL)(LLLR)\draw (LLL) -- (LLLR) node [midway, left=2pt, above=2pt] {\small{$x / 10$}};
        \Edge(LLLR)(LLLRL)\draw (LLLR) -- (LLLRL) node [midway, right=2pt,above=2pt] {\small{$x / 9$}};
        \Edge(LLLR)(LLLRR)\draw (LLLR) -- (LLLRR) node [midway, left=2pt, above=2pt] {\small{$y / 7$}};
        \Edge(LLLRR)(LLLRRR)\draw (LLLRR) -- (LLLRRR) node [midway, left=2pt, above=2pt] {\small{$y / 8$}};

        \Edge(U)(R)\draw (U) -- (R) node [midway, right=2pt] {\small{$y / 6$}};
        \Edge(R)(RR)\draw (R) -- (RR) node [midway, right=2pt] {\small{$y / 7$}};
        \Edge(RR)(RRR)\draw (RR) -- (RRR) node [midway, right=2pt] {\small{$y / 8$}};
        \Edge(R)(RL)\draw (R) -- (RL) node [midway, below=5pt, left=1pt] {\small{$x / 9$}};
        \Edge(RL)(RLU)\draw (RL) -- (RLU) node [midway, left=2pt] {\small{$x / 4$}};
        \Edge(RLU)(RLUL)\draw (RLU) -- (RLUL) node [midway, left=2pt] {\small{$x / 10$}};
        \Edge(RLU)(RLUR)\draw (RLU) -- (RLUR) node [midway, right=2pt] {\small{$y / 5$}};

          \end{tikzpicture}
    \caption{The tree $\textbf{U}(\overline{\widetilde{T}})$ for the pretzel of Figure \ref{fig:pretzellabel}, shown with both $X$- and $\mathcal{E}$-labellings and with all possible endpoints marked.}
    \label{fig:u}
    \end{figure}
\end{example}

We now proceed to show our main result, using the trees in $\textbf{U}(\overline{\widetilde{T}})$.

\begin{lem}\label{lem:conj1}
        Suppose $S$ is an $X$-tree such that idempaths go to leaves. Then $S$ admits a birooted $X$-graph morphism to some $\mathcal{U}_S \in \textbf{U}(\overline{\widetilde{S}})$.
    \end{lem}

    \begin{proof}
        Let $E$ be the edge set of $\overline{\widetilde{S}}$. We may consider $S$ as an $\mathcal{E}$-tree by fixing a map from $S$ to $\overline{\widetilde{S}}$ and labelling each edge by its image under this map. By definition, any $(\mathcal{U},v) \in \textbf{U}(\overline{\widetilde{S}})$ also has such a labelling, and under this labelling is deterministic.

        For any $(\mathcal{U},v) \in \textbf{U}(\overline{\widetilde{S}})$, we define a map $\phi_{(\mathcal{U},v)}: S \rightarrow (\mathcal{U},v)$ on the vertices of $S$ as follows. For $u \in V(S)$, consider the $\mathcal{E}$-label of the unique path from the start of $S$ to $u$. Since all idempaths in $S$ go to leaves, this path is almost simple in $\overline{\widetilde{S}}$, and hence corresponds to some path from the root in any $(\mathcal{U},v) \in \textbf{U}(\overline{\widetilde{S}})$ with the same $\mathcal{E}$-label. Moreover, since each $(\mathcal{U},v)$ is $\mathcal{E}$-deterministic, such a path is unique. Define $\phi_{(\mathcal{U},v)}(u)$ to be the vertex of $\mathcal{U}$ at the end of this path. Extend each $\phi_{(\mathcal{U},v)}$ to edges in the natural way.

        We claim that $\phi_{(\mathcal{U},v)}$ is a birooted $X$-graph morphism for some $(\mathcal{U},v) \in \textbf{U}(\overline{\widetilde{S}})$. By definition, all such maps respects adjacency of edges and $\mathcal{E}$-labels (and hence $X$-labels). Each also clearly sends $\alpha(S)$ to $\alpha((\mathcal{U},v))$.
        
        Consider the image of $\omega(S)$ under each $\phi_{(\mathcal{U},v)}$. Since the trunk of $S$ corresponds to an almost simple path in $\overline{\widetilde{S}}$, there is some $(\mathcal{U},v) \in \textbf{U}(\overline{\widetilde{S}})$ whose endpoint satisfies $\omega((\mathcal{U},v)) = \phi_{(\mathcal{U},v)}(v)$. This particular $\phi_{(\mathcal{U},v)}$ is therefore our desired birooted graph morphism and $\mathcal{U}_S := (\mathcal{U},v)$.\end{proof}

    \begin{cor}\label{cor:conj1}
        Let $S$ and $\mathcal{U}_S$ be as in Lemma \ref{lem:conj1}. If there exists a subgraph of $S$, rooted at $\alpha(S)$ and containing $\omega(S)$, which is isomorphic as a birooted $X$-graph to some $\mathcal{V} \in \textbf{U}(\overline{\widetilde{S}})$, then $\mathcal{V} \cong \mathcal{U}_S$.
    \end{cor}

    \begin{proof}
        By Lemma \ref{lem:conj1}, there exists some birooted $X$-graph morphism $\phi:S \to \mathcal{U}_S$. Suppose, for a contradiction, that $\mathcal{V} \ncong \mathcal{U}_S$. Since $\mathcal{U}_S,\mathcal{V} \in \textbf{U}(\overline{\widetilde{S}})$, they only differ in the location of their endpoint. Therefore a contradiction will be reached if we show that $\omega(\mathcal{V})$ is in the same location as $\omega(\mathcal{U}_S)$.
        
        Since $\omega(S) \in \mathcal{V}$, the unique path from the start vertex to the end vertex, say $\mathcal{E}$-labelled by $\ell$, is also the label of trunk of $\mathcal{V}$. By definition of $\phi$, $\omega(\mathcal{U})$ is at the end of the unique path in $\mathcal{U}$ from the start vertex labelled $\ell$. But this is exactly where the endpoint of $\mathcal{V}$ is, hence a contradiction.
    \end{proof}

    \begin{lem}\label{lem:conj2}
        Fix an $X$-tree $T$ and its idempath identified descendant $\widetilde{T}$. Let $S$ be an $X$-tree and suppose there exists a birooted $X$-graph morphism $\phi:S \rightarrow \widetilde{T}$. Then there exists a birooted $X$-tree $T_S$, with $[T_S]_\mathcal{M} = [T]_\mathcal{M}$, which contains an isomorphic copy of $S$ rooted at $\alpha(T_S)$, with $\omega(T_S)$ located at the natural image of the endpoint of $S$ under this embedding.
    \end{lem}

    \begin{proof}
        Consider the graph morphism $\sim: T \to \widetilde{T}$. Let $e$ be any edge of $S$. Since $\phi(e)$ is an edge of $\widetilde{T}$, there is a unique edge $e'$ of $T$ for which $\widetilde{e'} = \phi(e)$. Since $\phi$ and $\sim$ are graph morphisms, we have that \[
        \phi(\alpha(e)) = \alpha\left(\widetilde{e'}\right) = \widetilde{\alpha(e')} \textrm{ and }\phi(\omega(e)) = \widetilde{\omega(e')} \textrm{ and } \lambda_X(e) = \lambda_X(e').\]Our strategy is to build a copy of $S$ onto $T$ by copying the primed edges into their correct location, relying on Theorem \ref{thm:copy}. We do so top-down, assembling the `highest' cones of $S$ first and later moving them into position. Since we require the endpoint to be in our built $S$, we first ensure that $\omega(T)$ is in the correct location so we can copy it later.

        If $\omega(S) = \alpha(S)$, then we must have \[\widetilde{\alpha(T)} = \alpha\left(\widetilde{T}\right) = \phi(\alpha(S)) = \phi(\omega(S)) = \omega\left(\widetilde{T}\right) = \widetilde{\omega(T)}.\]Hence the start and end vertex of $T$ are connected by a constructable semiwalk by Theorem \ref{thm:whenidentify}. If they are trivially connected then set $T_0 := T$, otherwise set $T_0 := T|^{\omega(T) \to \alpha(T)}$.
        
        Otherwise, if $\omega(S) \neq \alpha(S)$, then there is some edge $e$ of $S$ with $\omega(e) = \omega(S)$. Via similar reasoning to the case above, we see that\[\widetilde{\omega(T)} = \omega\left(\widetilde{T}\right) = \phi(\omega(S)) = \phi(\omega(e)) = \widetilde{\omega(e')}.\]Hence $\omega(e')$ and $\omega(T)$ are connected by a constructable semiwalk by Theorem \ref{thm:whenidentify}. Set $T_0 := T|^{\omega(T)\to\omega(e')}$.

        By Theorem \ref{thm:copy}, $[T_0]_\mathcal{M} = [T]_\mathcal{M}$. We now build a sequence of trees $T_1,\dots,T_{h-1}$, where $h$ is the height of $S$, where each tree $T_i$ contains $T$ as a subtree rooted at $\alpha(T_i)$, and satisfies the following properties:\begin{enumerate}[start=0,label={(\bfseries P\arabic*)}]
        \item $[T_i]_\mathcal{M} = [T]_\mathcal{M}$.
        \item For any edge $e$ of $S$ with the height of $\mathrm{Cone}_S(\omega(e))$ at most $i$, $T_i$ contains a copy of the $\mathrm{Cone}_S(\omega(e))$ rooted at $\omega(e')$, with the endpoint in place if $\omega(S) \in \mathrm{Cone}_S(\omega(e))$.
        \item If $\omega(S) = \alpha(S)$, then $\omega(T_i) = \alpha(T_i)$. If instead $\omega(S) \neq \alpha(S)$, then for the edge $f$ of $S$ with $\omega(f) = \omega(S)$, if the height of $\mathrm{Cone}_S(\omega(f))$ is greater than $i$, then $\omega(T_i) = \omega(f')$.
        \end{enumerate}We construct our trees $T_i$ by induction. Our base case is the tree $T_0$ constructed above.

        By construction, $T_0$ immediately satisfies \textbf{P0}. For \textbf{P1}, note that any edge $e$ of $S$ with $\mathrm{Cone}_S(\omega(e))$ having height at most $0$, must be an edge which terminates at leaf. Hence $\mathrm{Cone}_S(\omega(e))$ is a single vertex, either with or without the endpoint of $S$. If it doesn't contain the endpoint, then certainly the cone of $\omega(e')$ in $T_0$ contains a copy of the trivial tree. If it does contain the endpoint, then our manipulation above ensured that $\omega(T_0) = \omega(e')$ as required: hence \textbf{P1} is satisfied. Moreover, our construction of moving the endpoint ensured that \textbf{P2} is satisfied by $T_0$ as the $\omega(T_0)$ is either $\alpha(T_0)$ (if $\omega(S) = \alpha(S)$), or at the end of the prime of the correct edge. Hence our base case is satisfied.

        Now suppose that we have constructed a tree $T_i$ satisfying our three required properties. We construct the tree $T_{i+1}$ as follows. Let $e$ be any edge of $S$ where $\mathrm{Cone}_S(\omega(e))$ has height $i+1$, and consider the edges $f_1,\dots,f_k$ of $S$ in this cone which are incident with $\omega(e)$. The height of each $\mathrm{Cone}_S(\omega(f_j))$ is at most $i$, and so (by inductive assumption) there is an isomorphic copy of $\mathrm{Cone}_S(\omega(f_j))$ located in $T_i$ at $\omega({f_j}')$, containing $\omega(T_i)$ if the respective cone in $S$ contains $\omega(S)$.

        Now for each $f_j$, note that \[\widetilde{\alpha({f_j}')} \ = \ \alpha\left(\widetilde{{f_j}'}\right) \ = \ \alpha(\phi(f_j)) \ = \ \phi(\alpha(f_j)) \ = \ \phi(\omega(e)) \ = \  \widetilde{\omega(e')}.\]Hence there is a constructable semiwalk on $T$ connecting $\alpha({f_j}')$ and $\omega(e')$. This semiwalk still exists in the embedded copy of $T$ in $T_i$, and so we appeal to Theorem \ref{thm:copy} to copy the $T_i$-cone of $\omega({f_j}')$ to $\omega(e')$, forming $T_i|^{\omega({f_j}')\to\omega(e')}$.

        Perform this operation for all of $f_1,\dots,f_k$. The new cone of $\omega(e')$ in the resulting tree certainly contains a copy of $\mathrm{Cone}_S(\omega(e))$, with endpoint in place if $\omega(S)$ is in $\mathrm{Cone}_S(\omega(f_j))$ for some $j$. If $\omega(S) = \omega(e)$, then since $T_i$ satisfied \textbf{P2} for the edge $e$ (since its $S$-cone has height $i+1 > i$), $\omega(T_i) = \omega(e')$, and we haven't moved it in constructing our new tree. Thus it is still located at $\omega(e')$. This shows that the new cone of $\omega(e')$ contains a copy of the $\mathrm{Cone}_S(\omega(e))$, including endpoint if $\omega(S) \in \mathrm{Cone}_S(\omega(e))$.

        Apply these operations for all such edges $e$ and call the resulting tree $T_{i+1}$. By the argument above, $T_{i+1}$ satisfies \textbf{P1}, and since all operations performed maintain the $\mathcal{M}$-value by Theorem \ref{thm:copy}, we have that $[T_{i+1}]_\mathcal{M} = [T_{i}]_\mathcal{M} = [T]_\mathcal{M}$, and so \textbf{P0} is satisfied. Moreover, $T_{i+1}$ satisfies \textbf{P2}, since if $\omega(S) = \alpha(S)$, then we have $\omega(T_i) = \alpha(T_i)$, and we haven't moved the endpoint in our construction. Thus $\omega(T_{i+1}) = \alpha(T_{i+1})$. If instead $\omega(S) \neq \alpha(S)$, then the only case in which we had moved the endpoint is if $\omega(S) \in \mathrm{Cone}_S(\omega(f_j))$ for some $f_j$, which has height at most $i$. Thus $T_{i+1}$ trivially satisfies \textbf{P2} if we moved the endpoint (as the cone will have height at most $i$), or if we didn't move the endpoint then $T_{i+1}$ trivially satisfies \textbf{P2} because $T_i$ does.

        Hence we may inductively construct the tree $T_{h-1}$ satisfying our three properties. We may now construct our tree $T_S$. Let $e_1,\dots,e_l$ be the initial edges of $S$. The height of the $\mathrm{Cone}_S(\omega(e_i))$ is at most $h-1$ for each $i$, and thus respective copies of these cones are located at $\omega({e_i}')$ in $T_{h-1}$ by induction (including endpoint if applicable). Then note that\[\widetilde{\alpha({e_i}')} = \phi(\alpha(e_i)) = \phi(\alpha(S)) = \widetilde{\alpha(T)}.\]Thus there are $n$-constructable semiwalks on $T$ connecting $\alpha(T)$ to each $\alpha({e_i}')$. These semiwalks still exist in $T_{h-1}$. As before, we use Theorem \ref{thm:copy} to copy the cones of each $\alpha({e_i}')$ to $\alpha(T_{h-1})$, moving the endpoint if the $S$-cone of $\omega(e_i)$ contains the endpoint of $S$. This constructs a copy of $S$ rooted at the start vertex of our new tree, with the endpoint in place if the endpoint of $S$ is not the start vertex of $S$. If the start vertex of $S$ is the endpoint, then by inductive assumption, $T_{h-1}$ had endpoint at $\alpha(T_{h-1})$, and the construction above did not move this endpoint. Thus our endpoint of our new tree is also at the start vertex; the `correct' place in $S$. This is our required tree $T_S$, which satisfies \textbf{P0} by construction; that is $[T_S]_\mathcal{M} = [T]_\mathcal{M}$.\end{proof}

        \begin{cor}\label{cor:conj2}
            Let $T$ be an $X$-tree. Fix a $(\mathcal{U},v) \in \textbf{U}(\overline{\widetilde{T}})$. Then there exists a tree $T_{(\mathcal{U},v)}$, with $[T_{(\mathcal{U},v)}]_\mathcal{M} = [T]_\mathcal{M}$, which contains an isomorphic copy of $(\mathcal{U},v)$ rooted at $\alpha(T)$.
        \end{cor}

        \begin{proof}
            The graph $\widetilde{T}$ contains (possibly many) images of the graph $\overline{\widetilde{T}}$; fix one such subgraph. The morphism $(\mathcal{U},v) \rightarrow \overline{\widetilde{T}}$ which sends an edge to its $\mathcal{E}$-label naturally admits a morphism $(\mathcal{U},v) \rightarrow \widetilde{T}$ via our chosen embedded subgraph. This morphism is also an $X$-graph morphism since it is a morphism of graphs labelled by $\mathcal{E}$. Hence a tree $T_{(\mathcal{U},v)}$ exists by Lemma \ref{lem:conj2}.
        \end{proof}

        \begin{thm}\label{thm:equalinM}
            Let $C$ be an $X$-generated right cancellative monoid and define $\mathcal{M}$ as above. Suppose $S$ and $T$ are trees such that $\overline{\widetilde{S}} = \overline{\widetilde{T}}$. Then $[S]_\mathcal{M} = [T]_\mathcal{M}$. 
        \end{thm}

        \begin{proof}
            Let $\textbf{U} = \textbf{U}(\overline{\widetilde{S}}) = \textbf{U}(\overline{\widetilde{T}})$ and fix a $(\mathcal{U},v) \in \textbf{U}$. By Corollary \ref{cor:conj2}, there exists a tree $S_{(\mathcal{U},v)}$ containing an embedded copy of $(\mathcal{U},v)$ at the root, such that $[S_{(\mathcal{U},v)}]_\mathcal{M} = [S]_\mathcal{M}$. By Corollary \ref{cor:idemtoleaves}, there further exists a tree $S_{(\mathcal{U},v)}'$, constructed from $S_{(\mathcal{U},v)}$ with $[S_{(\mathcal{U},v)}']_\mathcal{M} = [S_{(\mathcal{U},v)}]_\mathcal{M}$, in which idempaths go to leaves. By examining the proof of Corollary \ref{cor:idemtoleaves}, we note that the resulting tree $S_{(\mathcal{U},v)}'$ still contains a copy of $(\mathcal{U},v)$ at the root, as $(\mathcal{U},v)$ has the property that idempaths go to leaves. Moreover, since $[S_{(\mathcal{U},v)}']_\mathcal{M} = [S]_\mathcal{M}$, Theorem \ref{thm:quotienteasy} implies that $\overline{\widetilde{{S_{(\mathcal{U},v)}'}}} = \overline{\widetilde{S}}$.
            
            By Lemma \ref{lem:conj1} and Corollary \ref{cor:conj1}, $S_{(\mathcal{U},v)}'$ admits a birooted $X$-graph morphism $\phi$ to $(\mathcal{U},v)$. By Lemma \ref{lem:retractfrommorph}, it follows that the trees $S_{(\mathcal{U},v)}'$ and $(\mathcal{U},v)$ are equal in the free adequate monoid, and so $[S]_\mathcal{M} = [S_{(\mathcal{U},v)}']_\mathcal{M} = [(\mathcal{U},v)]_\mathcal{M}$.

            Dually, we have that $[T]_\mathcal{M} = [(\mathcal{U},v)]_\mathcal{M}$. Hence we have \[[S]_\mathcal{M} = [(\mathcal{U},v)]_\mathcal{M} = [T]_\mathcal{M}\]and the result is proved.
        \end{proof}

        \begin{cor}\label{cor:PTequalsM}
            Let $C$ be an $X$-generated right cancellative monoid and define $\mathcal{M}$ as above. Then $\textrm{PT}(C;X) \cong \mathcal{M}$.
        \end{cor}

        \begin{proof}
            Theorem \ref{thm:quotienteasy} and Theorem \ref{thm:equalinM} give that the maps induced by the generating sets $X$ in both directions are morphisms.
        \end{proof}

We close by discussing the nature of pretzel monoids as expansions of right cancellative monoids, and in particular the extent to which they are
natural analogues of Margolis-Meakin expansions \cite{mm:inv} in inverse semigroup theory.

First, it straightforward to check that the construction of a pretzel monoid from a right cancellative monoid is an example of an \textit{expansion} in the sense of Birget and Rhodes \cite{birget:expansion}, moreover it is the free idempotent-pure expansion in the sense of \cite{lawson:expansions}.

The presentation which we show in Corollary \ref{cor:PTequalsM} defines a pretzel monoid in terms of a right cancellative monoid $C$ is clearly the left adequate analogue of a well-known inverse monoid presentation for Margolis-Meakin expansions \cite[Corollary 2.9]{mm:inv}. {The respective geometric representation also bear a certain resemblance: elements of the Margolis-Meakin expansion of an $X$-generated group $G$ are birooted subgraphs of the Cayley graph of $G$ with respect to $X$; elements of the pretzel monoid of an $X$-generated right cancellative
monoid $C$ are not quite subgraphs of the Cayley graph of $C$ (this would require them to be \textit{co-deterministic}, and any definition attempting to enforce this
property is likely to result in a left ample monoid) but Proposition~\ref{prop:cayleygraphchunks} shows how they are closely related to such subgraphs.}

A Margolis-Meakin expansion of a group $G$ has maximal group image $G$. In contrast, the maximal right cancellative image of the Pretzel monoid
$\textrm{PT}(C;X)$ will clearly be the special right cancellative monoid $C'$ defined from $C$ as at the beginning of Section~\ref{sec:pretzel}, and so will
be $C$ only in the case that $C$ is itself special right cancellative.

Taking $C$ to be free, we recover $\mathrm{FLAd}(X)$, in a similar regard to how one recovers free inverse monoids from taking free groups in Margolis-Meakin expansions.

{Margolis-Meakin expansions are \textit{$E$-unitary} inverse monoids; the property of being $E$-unitary has several equivalent characterisations for inverse semigroups,
which turn out to be non-equivalent in the left adequate setting \cite{branco:lad, branco:ehresmannadequacy, gomes:lad}. One characterisation is that an inverse monoid
is $E$-unitary if the natural morphism to its maximal group image is \textit{idempotent-pure}, that is, if the elements which map to $1$ in the maximal group image are all idempotent. It is immediate from our presentation of pretzel monoids that $\textrm{PT}(C;X)$ has the corresponding properties with respect both to the right cancellative monoid $C$ used to construct it and the special right cancellative monoid $C'$ (which is its maximal right cancellative image). Moreover, it can easily be shown from the presentation that $\textrm{PT}(C;X)$ in the category of $X$-generated left adequate monoids with maximal right cancellative image $C'$, with morphisms the idempotent-pure $(2,1,0)$-morphisms.}

{It remains open if there is an alternative  natural (perhaps geometric) way to define a left adequate expansion of a right cancellative monoid $C$ which shares other properties of the Margolis-Meakin expansion, in particular such that the right cancellative image is always $C$.}


\bibliographystyle{plain}

\end{document}